\documentclass{amsart}
\usepackage[utf8]{inputenc}
\usepackage[english]{babel}
\usepackage{url}
\usepackage[dvips]{graphics,epsfig}
\usepackage{times}
\usepackage{amsmath,amsthm,amssymb}
\usepackage{array}
\usepackage{bbm}
\usepackage{xcolor}
\usepackage{verbatim}
\usepackage{upgreek}
\usepackage[foot]{amsaddr}

\usepackage[most]{tcolorbox}

\usepackage{geometry}

\def\span{\operatorname{span}}

\newtheorem{theorem}{Theorem}[section]

\newtheorem{lemma}[theorem]{Lemma}
\newtheorem{question}[theorem]{Question}
\newtheorem{cor}[theorem]{Corollary}
\newtheorem{prop}[theorem]{Proposition}
\newtheorem{fact}[theorem]{Fact}
\newtheorem{claim}[theorem]{Claim}

\newtheorem{conj}[theorem]{Conjecture}

\theoremstyle{definition}
\newtheorem{definition}[theorem]{Definition}

\theoremstyle{remark}
\newtheorem{remark}[theorem]{Remark}
\newcommand{\m}{\mathbb }

\def\Ind{\setbox0=\hbox{$x$}\kern\wd0\hbox to 0pt{\hss$\mid$\hss} \lower.9\ht0\hbox to 0pt{\hss$\smile$\hss}\kern\wd0}
\def\Notind{\setbox0=\hbox{$x$}\kern\wd0\hbox to 0pt{\mathchardef \nn=12854\hss$\nn$\kern1.4\wd0\hss}\hbox to 0pt{\hss$\mid$\hss}\lower.9\ht0 \hbox to 0pt{\hss$\smile$\hss}\kern\wd0}
\def\ind{\mathop{\mathpalette\Ind{}}}
\def\nind{\mathop{\mathpalette\Notind{}}}
\numberwithin{equation}{section}

\title{A counterexample to the stable forking conjecture}
\author{James Freitag and Scott Mutchnik}
\address{Department of Mathematics, Statistics, and Computer Science, University of Illinois at Chicago}
\email{jfreitag@uic.edu, mutchnik@uic.edu}
\thanks{SM was supported by the NSF under Grant No. DMS-2303034. (SM additionally received a grant of access to ChatGPT Pro through OpenAI's academic researcher support program, though the arguments in this paper were based on proofs generated using JF's ChatGPT account.) JF was supported by NSF grants No. DMS-2452197 and DMS-2554149.}

\begin{document}

\begin{abstract}
Using ChatGPT 5.6, we find a counterexample to the stable forking conjecture. This answers a long-standing open question of Hart, Kim and Pillay (1996). 
\end{abstract}

\maketitle

\section{Introduction}

The \textit{stable forking conjecture}, introduced by Hart, Kim and Pillay at the Fields Institute in 1996 (see \cite{Kim01}; the conjecture is also discussed early on in \cite{kim2001around}), represents a fundamental gap in our current understanding of the nature of simple theories. The success of simplicity theory as a research program in model theory is in large part due to the connections originally discovered by Kim \cite{Kim98} and Kim and Pillay \cite{KP99} between simplicity and stability. Stability is a stronger foundational property which was the predominant subject of pure model theory from the 1970s to the 1990s. Starting in the late 1990s, after Kim (\cite{Kim98}) and Kim and Pillay (\cite{KP99}) showed that the larger class of simple theories retains some of the useful properties of stable theories, simplicity theory became one of model theory's central research directions. Simplicity theory began the general program, which some call neostability theory, in which methods from stability theory have been adapted to more general settings. The stable forking conjecture is perhaps the most prominent open problem about the possibilities and limitations of such adaptations. 

Starting with Morley's categoricity theorem \cite{M65}, which says that  every first-order theory which is $\kappa$-categorical\footnote{has exactly one model of size $\kappa$} for some uncountable $\kappa$ is $\kappa'$-categorical for every uncountable cardinal $\kappa'$, classification theory originally concentrated on the question of determining the number of models of a first-order theory in each infinite cardinality. Stability is the property of a first-order theory that there does not exist a formula with the \textit{order property}, meaning a formula $\varphi(x, y)$ with a sequence $\{a_{i}, b_{i}\}_{i < \omega}$ such that $\models \varphi(a_{i}, b_{j})$ exactly when $i > j$. Stability was observed by Shelah \cite{Sh90} to be important to the more general problem of classifying first-order theories $T$ according to their numbers of models $I(T, \kappa)$ of cardinality $\kappa$ for each infinite cardinal $\kappa$. For uncountable cardinals $\kappa$, \textit{unstable} theories $T$ just have the maximum number of models of size $\kappa$: $I(T, \kappa)=2^{\kappa}$. By contrast, Shelah \cite{Sh90} proves the main result of a comprehensive model-counting classification for first-order theories (completed by Hart, Hrushovski and Laskowski in \cite{HHL00}) by introducing a structure theory for models of stable theories, based on his relation of \textit{forking independence} $\ind$.

Shelah's forking independence relation $A \ind_{C} B$, between subsets $A, B \subseteq M$ of a model $\mathbb{M} \models T$ of a theory over a subset $C \subset \mathbb{M}$, is a geometric phenomenon directly analogous to such genericity phenomena as linear disjointness in vector spaces or algebraic independence in fields. Specific incarnations of forking in various settings have sometimes played an important role in the development of those subjects (see e.g. differential algebra \cite{freitag2022any, freitag2023degree, freitag2017strong, nagloo2017algebraic}). There is a precise characterization of stable theories using properties of forking independence: forking independence is symmetric in stable theories, and stability is equivalent to stationarity of forking independence over models, or the uniqueness of nonforking extensions of types over models. The forking independence relation is now known to play an important role in the more general classification-theoretic research program initiated in \cite{Sh90}, \cite{She95}, and \cite{DS04} (see also \cite{SU08}). Much of the aim of modern classification theory is to classify unstable theories, including \textit{simple theories}. These theories, which omit formulas with the \textit{tree property} rather than the order property, are defined as follows:

\begin{definition}
    (Shelah, \cite{Sh90})

A theory is \textit{simple} if there is no formula $\varphi(x, y)$, $k < \omega$ and $\{b_{\eta}\}_{\eta \in \omega^{< \omega}}$ such that for each $\sigma \in\omega^{\omega}$, the set of formulas $\{\varphi(x,b_{\sigma|_n})\}_{n < \omega}$ is consistent, but for each $\eta\in\omega^{<\omega}$, $\{\varphi(x,b_{\eta^\smallfrown \langle i \rangle})\}_{i < \omega}$ is $k$-inconsistent.\footnote{A set of formulas is \textit{$k$-inconsistent} if any $k$ formulas in the set form an inconsistent set.} A non-simple theory is said to have the \textit{tree property}, or $\mathrm{TP}$.
\end{definition}

Kim \cite{Kim98} shows that forking independence is not only symmetric in stable theories, but is even symmetric in simple theories. Kim and Pillay \cite{KP99}, moreover, show that simplicity is \textit{equivalent} to symmetry of forking independence (just like stability is equivalent to stationarity of forking independence). They also prove that forking independence satisfies the \textit{independence theorem} over models in simple theories, a generalization of the stationarity property of forking in stable theories. Additionally, they give a characterization (Fact \ref{Kim-Pillay theorem} below) of simplicity as the existence of an \textit{abstract} independence relation satisfying properties similar to those of forking independence in stable theories. A relation satisfying these properties will itself coincide with forking independence. That independence in simple theories exhibits similar behavior to independence in stable theories demonstrates a deep connection between stability and simplicity, and lays the foundation for the broader program of using stability-theoretic phenomena to classify unstable theories more generally, including $\mathrm{NTP}_{2}$ theories and $\mathrm{NSOP}_{n}$ theories. The literature on independence in non-simple theories is extensive, but see \cite{CK09}, \cite{KR17}, \cite{NSOP2}, and \cite{SOPEXP} for some representative papers. 

The \textit{stable forking conjecture} makes an even deeper claim about the connection between simplicity and stability, asserting that forking independence in simple theories, already known to instantiate many of the properties characteristic of forking independence in stable theories, is actually determined by stable formulas themselves: formulas without the order property used to define instability.

\begin{conj}(\emph{Stable forking conjecture}, Hart, Kim and Pillay (1996), cited in \cite{Kim01}).

In a simple theory, let $a \nind_{C} b$. Then there is a formula $\varphi(x, \bar{b}) \in \mathrm{tp}(a / Cb)$ such that $\varphi(x, \bar{b})$ forks over $C$, and such that the parameter-free formula $\varphi(x, y)$ is stable (i.e. does not have the order property).

\end{conj}

Progress toward the stable forking conjecture has come from several sources:

\begin{itemize}
    \item Kim (\cite{Kim01}) shows that one-based theories with elimination of hyperimaginaries satisfy the conclusion of the stable forking conjecture. (See also Casanovas and Potier (\cite{casanovaspotier2018}), who show preservation of the conclusion of the stable forking conjecture in $T^{\mathrm{eq}}$. Additionally, see Brower and Hill \cite{BrowerHillHyperimaginaries}, who give a simplified proof that finite-rank supersimple one-based theories with weak elimination of hyperimaginaries satisfy the conclusion of the stable forking conjecture, which already follows from Kim's work and \cite{BuechlerPillayWagner2001}).

    \item Peretz (\cite{Per06}) prove the conclusion of the stable forking conjecture over an empty base $C = \emptyset$ between rank-two types in countably categorical supersimple theories.

    \item Brower (\cite{Brower2012}) improves Peretz's result using new techniques, removing the countable categoricity requirement and weakening the rank-two hypothesis on \textit{one} of the two types to being of any finite rank.

    \item Palacín and Wagner (\cite{PW13}) prove the conclusion of the stable forking conjecture in countably categorical, $\mathrm{CM}$-trivial supersimple theories.
    \item With John Baldwin, the authors of this paper (\cite{KoponenConjecture}) make progress on the \textit{simple Kim-forking conjecture} in $\mathrm{NSOP}_{1}$ theories.
\end{itemize}

Moreover, the second author (\cite{MutchnikReducing}) gives an \textit{existential} generalization of \textit{universal} results claimed by Peretz in rank three, reducing stability of the forking relation over a base to finitely many pregeometries in any fixed finite rank. 

Some additional motivation for the stable forking conjecture comes from consequences of its generalizations for other problems in classification theory:

\begin{itemize}
    \item Gomez (\cite{GomezRosyTheories}) shows that the \textit{stable Kim-forking conjecture} in $\mathrm{NSOP}_{1}$ theories implies that every $\mathrm{NSOP}_{1}$ rosy theory is simple, which would resolve a problem of \cite{D19} and \cite{K21}.

    \item The \textit{dependent forking conjecture} (or ``dependent dividing conjecture" as in \cite{Che14}) over sets in $\mathrm{NTP}_{2}$ theories implies, as observed by Chernikov in a personal communication as well as independently by the second author\footnote{Chernikov noted the requirement that the dependent dividing conjecture must hold over all sets, not just models}, that every $\mathrm{NTP}_{2}$ theory is resilient;  see the introduction to \cite{MutchnikReducing}. This would resolve a problem from \cite{BYC07}.

    \item The dependent dividing conjecture over models in $\mathrm{NTP}_{2}$ theories implies that $\mathrm{NSOP}_{n} \cap \mathrm{NTP}_{2} = \mathrm{simple}$, which would answer a question from \cite{Che14}; again, see the introduction to \cite{MutchnikReducing}.
\end{itemize}

See also Onshuus (\cite{Ons06}), Ealy and Onshuus (\cite{ealy2016thorn}), Palacín (\cite{palacin2012omega}), and Takahashi \cite{TakahashiConsequences}.

The progress that we will make on the stable forking conjecture in this paper is as follows:

\begin{theorem}
    The stable forking conjecture is false.
\end{theorem}

 The counterexample is the theory of an infinite-dimensional vector space over the division ring of fractions of the quantum graph algebra of the random graph (as in \cite{Goodearl2024Binomial}).

\subsection{Use of AI}
This is an AI-generated result proven with the help of GPT-5.6 Sol. Specifically, we prompted ChatGPT to construct a counterexample to the stable forking conjecture with a detailed series of prompts which took into account the many restrictions and likely restrictions such a counterexample would have to satisfy. For instance, we prompted it to use the strategy, standard in the literature, of proving the example is simple and characterizing forking independence by finding an abstract independence relation satisfying the criteria shown by Kim and Pillay (\cite{KP99}, Fact \ref{Kim-Pillay theorem}), to characterize independence in simple theories. Recent work of the second author in \cite{MutchnikReducing}, giving existential, higher-rank generalizations to results claimed by Peretz in \cite{Per06}, also highly constrains the type of counterexamples to stable forking which might exist in a finite rank setting, hence our search for an infinite rank counterexample. 

The proofs of the relevant facts around the counterexample as well as some aspects of the specific setup have been written entirely by the authors, and no AI tool was used in the actual writing of this manuscript. (AI was used for proofreading, with the edits performed and evaluated manually by the authors.) In retrospect, it seems unlikely that this counterexample would have been found in the near term without the use of generative AI, though model theoretic investigations of noncommutative rings in various languages are not without precedent in the recent literature \cite{itai2011note, itai2012first, itai2015notes, itai2017model}. 

The two sorted structure considered in this manuscript has a somewhat delicate setup - it involves a pair $(K,V)$ where $K$ is a division ring and $V$ is a module over $K$. As we will explain in the next section, there are two natural ways one normally considers such a structure. The first is to add each element of $K$ as a unary function (interpreted as scalar multiplication on the $V$ sort). The second is to consider $K$ itself in the language of rings with a binary function $K \times V \rightarrow V$ interpreting scalar multiplication. In the noncommutative setting, the first approach yields a stable structure, while the second would not have a simple theory. Perhaps the key general idea is to consider a third setup in between these two. The origins of this idea, for us, come from the thesis of Kasal \cite{kasal2010model}, which we describe in the following paragraphs. 

A \emph{derivation space} is a two-sorted structure $(V,K)$ in which $K$ is a field of characteristic zero and $V$ is a $K$-vector space. Both will be considered in their natural languages. Additionally, there is a function $$D: V \times K \rightarrow K$$ with the property that $D(v,-)$ is a derivation on $K$ for each $v \in V$. Additionally, the first coordinate of $D$ commutes with linear combinations in the natural way: $$D(\sum r_i v_i,x) = \sum r_i D(v_i,x).$$

In a certain expansion of this setup, Kasal obtains a model companion for the theory. Crucially, the ring of derivations generated by composition is not considered as a sort (it is not interpretable). However, the derivation corresponding to any element of the ring is a (non-uniformly) definable derivation on $K$ in the structure. This is analogous - sort of dual - to our example, in which we only have access to the generators of the division ring, rather than the entire ring structure.

\subsection{Motivating examples}

A main \textit{stable} example from the literature   motivating our counterexample to the stable forking conjecture is the theory of an infinite-dimensional vector space over an algebraically closed field, such as the field of complex numbers $\mathbb{C}$. This vector space example can be found in, say, \cite{P96}. 

We work in the two-sorted structure $(K, V)$ in the field sort $K$ and the vector space sort $V$, in the language with field addition and multiplication, vector space addition, and scalar multiplication. Sets within models of this theory, say a sufficiently saturated ambient model $\mathbb{M}$, come with a closure operator: we may define $K(A)$-vector and $K(B)$-vector spaces $(K(A), V(A)) \subseteq (K(B), V(B))$ to satisfy the strong inclusion relation $(K(A), V(A)) \leq (K(B), V(B))$ if $K(A)$ is algebraically closed in $K(B)$, and some (equivalently every) $K(A)$-basis for $V(A)$ remains independent over $V(B)$. This is similar to the relation defined by the conditions of Proposition \ref{strong inclusion} below. Then the closure $\mathrm{cl}(A)$ for $A \subseteq \mathbb{M}$ is the smallest $A' \subseteq \mathbb{M}$ with $A \subseteq A'$ and $A' \leq \mathbb{M}$ (as in Definition \ref{closure with respect to strong inclusion}).\footnote{This will coincide with the usual model-theoretic algebraic closure, though we will not need to explicitly refer to the model-theoretic algebraic closure in this paper.} The theory has embedding and quantifier elimination properties with respect to the closure as in Lemma \ref{embedding property for strong inclusions} and Proposition \ref{closure with respect to strong inclusion}. Crucially, independence in this theory, up to the closure operator, is determined by independence of algebraically closed fields and global linear disjointness in the vector space sort: $A \ind_{C} B$ if and only if $K(\mathrm{cl}(AC)) \ind^{\mathrm{ACF}}_{K(\mathrm{cl}(C))} K(\mathrm{cl}(BC))$ and $\mathrm{Span}_{K(\mathbb{M})}(V(\mathrm{cl}(AC))) \cap \mathrm{Span}_{K(\mathbb{M})}(V(\mathrm{cl}(BC))) = \mathrm{Span}_{K(\mathbb{M})}(V(\mathrm{cl}(C))) $.

A related, implicitly familiar example of a stable theory is the theory of an infinite-dimensional vector space over the field of fractions of a ring of $\mathbb{Q}$-polynomials in infinitely many variables, such as the field of fractions $\mathbb{Q}(X_{1}, \ldots, X_{n}, \ldots )$ of the polynomial ring $\mathbb{Q}[X_{1}, \ldots, X_{n}, \ldots ]$ in variables $\{X_{i}\}^{\infty}_{i = 1}$. (Like the previous example, this is in fact superstable.) A model of this theory is \textit{not} viewed as a structure in the full field sort. Rather, it is viewed as a two-sorted structure $(S, V)$ where $V$ is the vector space sort and $S$ is a sort indexing \textit{generators} of the polynomial ring, in the language with vector space addition and scalar multiplication of a vector $v \in V$ by a \textit{generator} $X_{s}$ indexed by $s \in S$. However, similarly to the discussion in Section \ref{Tnegstfk}, we can canonically recover the full $Q_{S} : = \mathbb{Q}(\{X_{s}\}_{s\in S})$-vector space structure from the structure on a vector space considered in this language. Similarly to the example of the theory of an infinite-dimensional vector space over an algebraically closed field, we obtain a strong inclusion relation $\leq$ as in Section \ref{Tnegstfk}, a closure operator $\mathrm{cl}$ as in Definition \ref{closure with respect to strong inclusion}, and embedding and quantifier elimination properties as in Lemma \ref{embedding property for strong inclusions} and Proposition \ref{closure with respect to strong inclusion}. The theory will be stable, and forking independence will coincide with global linear disjointness of closures:  $A \ind_{C} B$ if $\Gamma(A) \cap \Gamma(B) = \Gamma(C)$ and $\mathrm{span}_{Q_{S(\mathbb{M})}}(V(\mathrm{cl}(AC)))\cap\mathrm{span}_{Q_{S(\mathbb{M})}}(V(\mathrm{cl}(BC))) = \mathrm{span}_{Q_{S(\mathbb{M})}}(V(\mathrm{cl}(C))) $.

The main idea of our counterexample is that the commutative field in this stable example can be replaced by a \textit{noncommutative} division ring so that forking independence continues to coincide with this global linear disjointness relation, but where the global linear disjointness relation itself is no longer a stable relation.\footnote{The forking relation $\nind_{\emptyset}$, which is not necessarily definable, is defined to be stable if there is no indiscernible instance $\{a_{i}, b_{i}\}_{i< \omega}$ of the order property, so where $a_{i} \nind_{\emptyset} b_{j}$ exactly when $i < j$. See \cite{PW13} for this definition. As made explicit in Remark 7.5 of \cite{KoponenConjecture}, the conclusion of the stable forking conjecture requires the forking relation to be stable.} The noncommutativity of the division ring ends up encoding the order property in the intersection properties of closed subspaces. 

Here, instead of an infinite-dimensional vector space over the fraction field of a polynomial ring  over an infinite set of generators, a model of our theory that is a counterexample to the stable forking conjecture will be an infinite-dimensional module over the fraction division ring of the noncommutative quantum algebra associated with the random (bipartite) graph, as recalled in Subsection \ref{quantum division algebra}. Just as with global linear disjointness in our polynomial-ring example, global linear disjointness in this theory satisfies the required stability-generalizing properties from the Kim-Pillay theorem (Fact \ref{Kim-Pillay theorem}) for the relation to coincide with forking independence and the theory to be simple. The reason why the theory is not stable is that global linear disjointness is no longer stationary, due to the random graph structure. However, in this replacement of a commutative field with a noncommutative division ring, we lose stability of global linear disjointness. So informally speaking, the reason why we obtain a counterexample to the stable forking conjecture is not that we obtain a particularly unusual characterization of forking independence in a simple theory. Rather, we obtain a simple theory where the forking relation is familiar, characterized by global linear disjointness. But the familiar relation of \textit{global linear disjointness} now has the (at least formerly) unusual feature of representing the order property, whose absence was taken for granted over commutative fields.

\subsection{Organization}

In Section \ref{preliminaries} we introduce the necessary algebra for defining and proving quantifier elimination properties in the theory that will serve as our counterexample. We will start with Subsection \ref{quantum division algebra}, where we introduce the quantum division algebra of a graph from the literature on quantum algebra (\cite{Goodearl2024Binomial}). We will review the proof, following the proofs of well-known facts in noncommutative noetherian ring theory (found in a standard reference such as \cite{mcconnell1987noncommutative}), that the quantum division algebra of a graph has a division ring of fractions. We will conclude by proving a fact, which will be useful later, yielding a least set of vertices necessary to define any element of this division ring (Corollary \ref{least support}). Then, in Subsection \ref{modules over division rings}, we will review the necessary module theory over division rings, in particular showing the robustness of the definition of strong inclusion of modules over division rings. In Section \ref{Tnegstfk}, we will define the theory $T^{\neg \mathrm{StFk}}$ that we will show to be a counterexample to the stable forking conjecture, the theory of an infinite-dimensional module over the fraction division ring of the random bipartite graph's quantum algebra. Particularly, we will introduce the language with scalar multiplication by generators and show that the full module structure can be recovered from the structure in this language. We will then prove the embedding property for $T^{\neg \mathrm{StFk}}$ with respect to strong inclusion, show the existence of the closure of a set, and show that the theory $T^{\neg \mathrm{StFk}}$ satisfies quantifier elimination up to closure. Finally, in Section \ref{counterexample verification}, we will verify that $T^{\neg \mathrm{StFk}}$ is a counterexample to the stable forking conjecture. In Subsection \ref{simplicity of the example}, we will show that global linear disjointness of closures satisfies the Kim-Pillay axioms. This will prove simplicity (in fact, supersimplicity) of $T^{\neg \mathrm{StFk}}$ and confirm that forking-independence coincides with global linear disjointness. By simplicity, $T^{\neg \mathrm{StFk}}$ will be a candidate for being a counterexample to the stable forking conjecture. In Subsection \ref{instability of the independence relation} we will show that $T^{\neg \mathrm{StFk}}$ is in fact a counterexample: global linear disjointness, and thus the forking relation, will be an unstable relation. Finally, in a sidebar to Section 4, we will show that $T^{\neg\mathrm{StFk}}$ is even a counterexample to Ziegler's question from \cite{casanovaspotier2018}: the claim that the conclusion of the stable forking conjecture in one variable implies the conclusion of the full stable forking conjecture, in any number of variables.

\section{Algebraic preliminaries}\label{preliminaries}

A model of our counterexample to the stable forking conjecture will be a module over a noncommutative division ring with a generating set indexed by the vertices of a bipartite graph, in the language with symbols for module addition, scalar multiplication \textit{by members of the generating set}, and the graph relation on the generating set. We first describe the family of rings that this module's ring of scalars will belong to.

\subsection{The quantum division algebra of a bipartite graph}\label{quantum division algebra}

Let $A = I_{A} \sqcup J_{A}$ be a finite bipartite graph with edge relation $E_{A}$, with respect to the partition into sorts $I_{A}$, $J_{A}$. We define a noncommutative (associative) ring $R_{A}$ as follows: define its underlying set $R_{A} : = \mathbb{Q}\langle\{ X_{v} \}_{v \in A} \rangle / V$ to be the $\mathbb{Q}$-vector space quotient of the (associative) polynomial $\mathbb{Q}$-algebra $\mathbb{Q}\langle\{ X_{v} \}_{v \in A}\rangle$ in \textit{noncommuting} variables $X_{v}$ indexed by vertices $v \in A$, where $V$ is the left-right ideal of $\mathbb{Q}\langle\{ X_{v} \}_{v \in A} \rangle $ consisting of the $\mathbb{Q}$-subspace spanned by elements of $\mathbb{Q}\langle\{ X_{v} \}_{v \in A} \rangle $ of the form $r s t $, where $r, t \in \mathbb{Q}\langle\{ X_{v} \}_{v \in A} \rangle $ and $s$ is of one of the following forms:

\begin{itemize}
    \item For $i, i' \in I_{A}$, $s = X_{i}X_{i'}-X_{i'}X_{i}$.
    
    \item For $j, j' \in J_{A}$, $s = X_{j}X_{j'}-X_{j'}X_{j}$.

    \item For $i \in I_{A}$ and $j \in J_{A}$ with $ E_{A}(i, j)$, $s = X_{i}X_{j} - X_{j}X_{i}$.

    \item For $i \in I_{A}$ and $j \in J_{A}$ with $ \neg E_{A}(i, j)$, $s = X_{i}X_{j} + X_{j}X_{i}$.
\end{itemize}

Then because $V$ is a left-right ideal of $\mathbb{Q}\langle\{ X_{v} \}_{v \in A} \rangle $, the multiplication on $\mathbb{Q}\langle\{ X_{v} \}_{v \in A} \rangle $ yields a well-defined multiplication on $R_{A}$. (This is the quantum affine space $\mathcal{A}_{\mathbf{q}}$ as in \cite{Goodearl2024Binomial}.) We give a basis for $R_{A}$ over $\mathbb{Q}$:

\begin{fact}\label{ordered basis}
Let $A = (I_{A} \sqcup J_{A}, E_{A})$ be a finite bipartite graph, and let $v_{1}, \ldots, v_{n}$ be a fixed enumeration of the vertices of $A$. For $\bar{i} = (i_{1}, \ldots, i_{n}) \in \omega^{n} $, let $X^{\bar{i}} =: X_{v_{1}}^{i_{1}} \ldots X_{v_{n}}^{i_{n}}$. Then (the $V$-cosets of) the $X^{\bar{i}}$, where $\bar{i}$ ranges over $\omega^{n}$, form a basis for $R_{A}$ over $\mathbb{Q}$.
\end{fact}

\begin{proof}
    Let $M$ be the subspace of $\mathbb{Q}\langle\{ X_{v} \}_{v \in A})$ spanned by the $X^{\bar{i}} = X_{v_{0}}^{i_{1}} \ldots X_{v_{n}}^{i_{n}}$ for $\bar{i} \in \omega^{n}$. It suffices to define a $\mathbb{Q}$-linear map from $R_{A}$ to $M$ such that for $\bar{i} \in \omega^{n}$, the $V$-coset of $X^{\bar{i}}$ is mapped to $X^{\bar{i}}$. Let us define an associative $\mathbb{Q}$-algebra structure on $M$. (The $\mathbb{Q}$-linear map we will define from $R_{A}$ to $M$ will be a map of associative $\mathbb{Q}$-algebras.) For $1 \leq i < j \leq n$, define $\eta_{i, j}$ to be $-1$ if $v_{i} \in I$, $v_{j} \in J$, or $v_{i} \in J$ and $v_{j} \in I$, and $\neg E_{A}(v_{i}, v_{j})$, and $1$ otherwise. For $\bar{i}, \bar{j}$ in $\omega^{n}$, define 

$$\eta(\bar{i}, \bar{j}) := \prod_{1 \leq \ell < k \leq n} \eta^{i_{k} j_{\ell}}_{k\ell}$$

for $\bar{i}=(i_{1}, \ldots, i_{n})$, $\bar{j}=(j_{1}, \ldots, j_{n})$.

Then define a $\mathbb{Q}$-algebra structure on $M$ by $X^{\bar{i}}\cdot X^{\bar{j}} = \eta(\bar{i}, \bar{j}) X^{\bar{i}+\bar{j}}$. Let us show that this $\mathbb{Q}$-algebra structure is associative. For this it suffices to show that for $\bar{i}, \bar{j}, \bar{k} \in \omega^{n}$, $(X^{\bar{i}}\cdot X^{\bar{j}}) \cdot X^{\bar{k}} = X^{\bar{i}} \cdot (X^{\bar{j}}\cdot X^{\bar{k}}) $. Computing each side, for this it suffices to show that $\eta(\bar{i}, \bar{j})\eta(\bar{i}+\bar{j}, \bar{k})X^{\bar{i}+\bar{j} + \bar{k}}= \eta(\bar{i}, \bar{j}+\bar{k})\eta(\bar{j}, \bar{k}) X^{\bar{i}+\bar{j} + \bar{k}}$. Equivalently, it suffices to show that the coefficients are equal: $\eta(\bar{i}, \bar{j})\eta(\bar{i}+\bar{j}, \bar{k})= \eta(\bar{i}, \bar{j}+\bar{k})\eta(\bar{j}, \bar{k}) $. But $\eta(\bar{x}, \bar{y})$ is multiplicative in both $\bar{x}$ and $\bar{y}$. So both sides are equal to $\eta(\bar{i}, \bar{j}) \eta(\bar{j}, \bar{k})\eta(\bar{i}, \bar{k})$, proving the desired equation, and thus the associativity of the $\mathbb{Q}$-algebra structure on $M$.

Now define a $\mathbb{Q}$-linear map from $\mathbb{Q}\langle\{ X_{v} \}_{v \in A}\rangle$ to $M$ sending an arbitrary monomial $X_{u_{1}} \ldots X_{u_{N}}$ for vertices $u_{1}, \ldots u_{N} \in A$ to the product $X_{u_{1}} \cdot \ldots \cdot X_{u_{N}}$ in $M$; this map is well-defined by associativity of the $\mathbb{Q}$-algebra structure defined on $M$ (and is also a map of $\mathbb{Q}$-algebras). Note that for distinct vertices $u, v \in A$, $X_{u} \cdot X_{v}=-X_{v}X_{u}$ if $v \in I$, $u \in J$, or $v \in J$ and $u \in I$, and $\neg E_{A}(v, u)$, and $X_{u} \cdot X_{v}=X_{v}X_{u}$ otherwise. So $V$ is contained in the kernel of this map, yielding a map from $R_{A} = \mathbb{Q}\langle\{ X_{v} \}_{v \in A}\rangle/V$ to $M$ (which is likewise a map of $\mathbb{Q}$-algebras). For $\bar{i} \in \omega^{n}$, the $V$-coset of $X^{\bar{i}}$ is mapped to $X^{\bar{i}}$, as desired.
\end{proof}

By Fact \ref{ordered basis}, $R_{A}$ is a domain. This is analogous to the case of polynomial rings. Fixing an enumeration $v_{1}, \ldots, v_{n}$ of $A$, let the \textit{degree} $\mathrm{deg}(s)$ of some nonzero $s=\sum q_{\bar{i}} X^{\bar{i}} \in R_{A}$, with $\bar{i}$ ranging over those $\bar{i} \in \omega^{n}$ with nonzero coefficient $q_{\bar{i}} \in \mathbb{Q}_{\neq 0}$, be the lexicographically greatest such $\bar{i}$. Then if $s, t \in R_{A}$ are nonzero, $\mathrm{deg}(st)=\mathrm{deg}(s) + \mathrm{deg}(t)$ is defined, so $st$ is nonzero and $R_{A}$ is a domain.

Commutative domains have fields of fractions, but this construction does not generalize to general noncommutative domains. However, the construction does generalize to the specific case of $R_{A}$, which does have a \textit{division ring of fractions}: in this context, a \textit{(right) division ring} of fractions for a noncommutative associative domain $R$ is a noncommutative $R$-algebra which

\begin{itemize}
    \item is a division ring, and
    \item is such that every element can be written in the form $rs^{-1}$, for $r \in R$, $s \in R \backslash \{0\}$.
\end{itemize}

See \cite{mcconnell1987noncommutative}, 2.1.3 for a more general definition of the right quotient ring (or localization) with respect to a multiplicative subset. A right division ring of fractions $Q$ of a noncommutative associative domain $R$ is universal among division rings extending $R$: every homomorphism of $R$ into a division ring extends uniquely to all of $Q$ (\cite{mcconnell1987noncommutative}, Lemma 2.1.4\footnote{This proof of this lemma is stated to be ``straightforward," but doesn't seem that way to us. For a homomorphism $\varphi: R \hookrightarrow Q'$ of $R$ into a division ring $Q$, and $Q$ a right division ring of fractions of $R$, we want to extend $\varphi$ to all of $Q$ by defining $\varphi(rs^{-1})=\varphi(r)\varphi(s)^{-1}$ for $r \in R$, $s \in R_{\neq 0}$; we must show this extension of $\varphi$ is a well-defined homomorphism. That $R$ has a right division ring of fractions implies that it satisfies the \textit{right Ore condition} (\cite{mcconnell1987noncommutative}, 2.1.6), namely that any two $a, b \in R_{\neq 0}$ have a common nonzero right multiple: in $Q$, $b^{-1}a = cd^{-1}$ for $c, d \in R_{\neq0}$, so $ad = bc$ is a common nonzero right multiple of $a$ and $b$. We can now show that the extension of $\varphi$ to $Q$ is well-defined: suppose $rs^{-1} = tu^{-1}$ for $r, t \in R$, $s, u \in R^{\neq 0}$. By the right Ore condition, find some nonzero $c, d \in R$ such that $sc = ud$. By $rc (sc)^{-1}=rs^{-1} = tu^{-1}= td(ud)^{-1}$, this implies that $rc=td$. So $\varphi(r)\varphi(s)^{-1} = \varphi(rc)\varphi(sc)^{-1}= \varphi(td)\varphi(ud)^{-1}=\varphi(t)\varphi(u)^{-1}$, as desired. Now to show that $\varphi$ is additive, any pair $rs^{-1}, tu^{-1}$ for $r, t \in R$, $s, u \in R^{\neq 0}$ can be written such that $s = u$ using the Ore condition, and additivity is straightforward in this case. Similarly, any pair $rs^{-1}, tu^{-1}$ for $r \in R$, $s, t, u \in R^{\neq 0}$ can also be written such that $s = t$ using the Ore condition, and multiplicativity is then straightforward. Uniqueness of the extension $\varphi$ is also routine, as in commutative algebra.}). By symmetry, a left division ring of fractions, where the second bullet point is replaced with the condition that every element can be written in the form $s^{-1}r$, for $r \in R$, $s \in R \backslash \{0\}$, also satisfies this universal property. So a right division ring of fractions is also a left division ring of fractions, and we can eliminate the adjectives ``left" and ``right"; moreover, every division ring of fractions of a noncommutative associative domain $R$ is unique up to unique isomorphism over $R$ (corollary to Lemma 2.1.4, \cite{mcconnell1987noncommutative}).

That $R_{A}$ has a quotient division ring is standard in the literature, and can be proven abstractly: by Fact 2.1, $R_{A}$ can be constructed from $\mathbb{Q}$ by an iterated series of \textit{skew polynomial rings} (\cite{mcconnell1987noncommutative}, 1.2.3), so is (left and right) Noetherian by repeated applications of the Hilbert basis theorem (\cite{mcconnell1987noncommutative}, 1.2.9.iv). (See also §2.6 of \cite{mcconnell1987noncommutative} for a discussion of skew polynomial rings in multiple variables.) Because $R_{A}$ is a right Noetherian domain, it satisfies the right Ore condition (\cite{mcconnell1987noncommutative}, 2.1.15; see the previous footnote). By \cite{mcconnell1987noncommutative}, 2.1.12, that $R_{A}$ satisfies the right Ore condition implies that $R_{A}$ has a division ring of fractions. (In fact, by \cite{mcconnell1987noncommutative}, 2.1.6, having a division ring of fractions is equivalent to the right Ore condition; see the previous footnote.) However, we will present a more self-contained exposition of the proof using Laurent series that $R_{A}$ has a division ring of fractions; while we will have to review the proofs of the Noetherianity and the Ore condition from \cite{mcconnell1987noncommutative}, using Laurent series will simplify the hardest part of the proof, showing the existence of a division ring of fractions for $R_{A}$.

We saw in Fact \ref{ordered basis} that for $v_{1}, \ldots, v_{n}$ a fixed enumeration of the vertices of a finite bipartite graph $A$, $R_{A}$ has basis $X^{\bar{i}} =: X_{v_{1}}^{i_{1}} \ldots X_{v_{n}}^{i_{n}}$ for $\bar{i} \in \omega^{n}$, with multiplication given by $X^{\bar{i}}X^{\bar{j}}= \eta(\bar{i}, \bar{j})X^{\bar{i}+\bar{j}}$. This multiplication can be extended to the skew Laurent polynomial ring with basis $X^{\bar{i}} =: X_{v_{0}}^{i_{1}} \ldots X_{v_{n}}^{i_{n}}$ for $\bar{i} \in \mathbb{Z}^{n}$, where multiplication is given by the same formula $X^{\bar{i}}X^{\bar{j}}= \eta(\bar{i}, \bar{j})X^{\bar{i}+\bar{j}}$. It can then be completed  to the (iterated) skew Laurent series ring consisting of series $\sum_{\bar{i} \in \mathbb{Z}^{n}} q_{\bar{i}}X^{\bar{i}}$ with coefficients $q_{\bar{i}} \in \mathbb{Q}$, where the condition on which $q_{\bar{i}}$ may be nonzero depends on the enumeration on the vertices, and is defined inductively: as the base case, the constants from $\mathbb{Q}$ are in the Laurent series ring. And for $k < n$, having defined which series with $X_{v_{i}}$ only appearing for $i \leq k$ are in the Laurent series ring, we define which series with $X_{v_{i}}$ only appearing for $i \leq k+1$ are in the Laurent series ring: these are just those series where the exponents of $X_{v_{k+1}}$ have a lower bound, and the coefficient of each power of $X_{v_{k+1}}$, a series with $X_{v_{i}}$ only appearing for $i \leq k$, is in the Laurent series ring. This condition implies that multiplication of these series is well-defined according to distributivity, as with usual iterated Laurent series rings. (See, say, \cite{Xin2004} for some prior discussion in the literature on iterated laurent series rings.) For an induced subgraph $B \subset A$, $R_{B}$ embeds canonically into $R_{A}$ with basis those $X^{\bar{i}}$ with $(i_{1}, \ldots i_{n}) = \bar{i} \in \omega^{n}$ where $i_{k} =0$ for each $ v_{k} \in A \backslash B$. (As an embedding of $R_A$ into $R_B$, this does not depend on how the vertices are enumerated; note again that the Laurent series construction, on the other hand, depends on the enumeration of the vertices.) And the ring formed by applying the skew Laurent series construction to $B$ (with the induced enumeration) embeds canonically as an $R_{B}$-algebra into the ring formed by applying this construction to $A$, as the ring of elements $\sum_{\bar{i}\in \mathbb{Z}} q_{\bar{i}}X^{\bar{i}}$ where for $(i_{1}, \ldots i_{n}) = \bar{i} \in \mathbb{Z}^{n}$, $q_{\bar{i}}$ is nonzero only if $i_{k} =0$ for $ v_{k} \in A \backslash B$.

Just as \cite{cohn1995skew} observes for skew Laurent series rings in one variable, that the skew Laurent series ring corresponding to $A$ is a division ring is just as in the commutative case. By induction on the size of $A$, the skew Laurent series ring corresponding to $A \backslash \{{v_{n}} \}$ is a division ring. So in the skew Laurent series ring for $A$, every nonzero element can be written as a product of a unit of the form $rX^{k}_{n}$, where $r$ is in the skew Laurent series ring for $A \backslash \{{v_{n}} \}$, and an element of the form $1 - q$, where every monomial of $q$ with nonzero coefficient has positive exponent for $X_{n}$. It suffices to show $1 - q$ is invertible. But $\sum^{\infty}_{\ell=0} q^{\ell}$ is well-defined: the coefficient for each $X^{n}$ will be an element of the skew Laurent series ring for $A \backslash \{{v_{n}} \}$ consisting of a sum of elements of that ring taken from the expansions of finitely many of the $q^{\ell}$. As in the standard argument from commutative algebra, this is an inverse for $1 - q$. So this skew Laurent series ring is a division ring.

We next show that $R_{A}$ is left noetherian (there is no strictly ascending chain of left ideals; this is the Hilbert basis theorem (\cite{mcconnell1987noncommutative}, 1.2.9.iv)); by symmetry, $R_{A}$ will also be right noetherian. By induction, $R_{A \backslash \{v_{n}\}}$ is noetherian; we view elements of $R_{A}$ as polynomials $\sum^{N}_{k=1} q_{k}X^{k}_{n}$ with $q_{k} \in R_{A \backslash \{v_{n}\}}$. Let $I_{1} \subseteq \ldots \subseteq I_{i} \subseteq I_{i+1} \subseteq \ldots$ be an ascending chain of ideals. From how the multiplication is defined in $R_{A}$, for each fixed $i$ the sets of leading coefficients $q_{N}$ of elements $\sum^{N}_{k=1} q_{k}X^{k}_{n} \in I_{i}$ of degree $N$ as a polynomial in $X_{n}$ form an ascending chain $I^{1}_{i} \subseteq \ldots \subseteq I^{N}_{i} \subseteq I^{N+1}_{i} \subseteq \ldots$ of left ideals of $R_{A \backslash \{v_{n}\}}$ in the index $N$. By Noetherianity of $R_{A \backslash \{v_{n}\}}$, the ascending chain $I^{1}_{1} \subseteq \ldots \subseteq I^{i}_{i} \subseteq I^{i+1}_{i+1} \subseteq \ldots$ terminates, say at $i = \ell$. Then for $i \geq \ell$, $I^{1}_{i} \subseteq \ldots \subseteq I^{N}_{i} \subseteq I^{N+1}_{i} \subseteq \ldots$ terminates at $N = \ell$. And $I^{\ell}_{1} \subseteq \ldots \subseteq I^{\ell}_{i} \subseteq I^{\ell}_{i+1} \subseteq \ldots$ terminates at $i = \ell$. Let us choose $i \geq \ell$ large enough such that, for the finitely many ascending chains $I^{0}_{1} \subseteq \ldots \subseteq I^{0}_{i} \subseteq I^{0}_{i+1} \subseteq \ldots; I^{1}_{1} \subseteq \ldots \subseteq I^{1}_{i} \subseteq I^{1}_{i+1} \subseteq \ldots; \ldots ; I^{j}_{1} \subseteq \ldots \subseteq I^{j}_{i} \subseteq I^{j}_{i+1} \subseteq \ldots ; \ldots ; I^{\ell-1}_{1} \subseteq \ldots \subseteq I^{\ell-1}_{i} \subseteq I^{\ell-1}_{i+1} \subseteq \ldots$, these ascending chains also terminate at $i$. We then conclude that $I^{j}_{i}= I^{j}_{i+1}$ for all $j \geq 0$. But this implies $I_{i}=I_{i+1}$: choose $q \in I_{i+1} \backslash I_{i}$ of least degree in $X_{n}$. Then there is $r \in I_{i}$ with the same leading coefficient. Then $q - r \in I_{i+1} \backslash I_{i}$, but $q-r$ has degree in $X_{n}$ less than that of $q$, a contradiction. So $I_{i}=I_{i+1}$, and we conclude that the ascending chain $I_{1} \subseteq \ldots \subseteq I_{i} \subseteq I_{i+1} \subseteq \ldots$ terminates.

But any right Noetherian domain satisfies the right Ore condition, as shown in \cite{mcconnell1987noncommutative}, 2.1.15: any two nonzero $a, b$ have a common nonzero right multiple. This is proven as follows: by Noetherianity, there is some $N$ such that the right ideal generated by the first $N$ terms of the sequence $\{b^{i}a\}^{\infty}_{i=0}$ is equal to the right ideal generated by the first $N+1$ terms. Then $b^{N}a= \sum^{N-1}_{i =0} b^{i}ar_{i}$, for some $r_{i}$. Since $b$ and $a$ are nonzero, $b^{N}a$ is nonzero, so one of the $r_{i}$ for $0 \leq i \leq N-1$ is nonzero. Let $i^{*}$ be the least $i$ for which $r_{i}$ is nonzero. Then $b^{i^{*}}a r_{i} = -b^{N}a + \sum^{N-1}_{i =i^{*}+1} b^{i}ar_{i}$, and by cancellativity, $a r_{i} = -b^{N-i^{*}}a + \sum^{N-1}_{i =i^{*}+1} b^{i-i^{*}}ar_{i}$. But the left side is a nonzero right multiple of $a$, and the right side is a nonzero right multiple of $b$, as desired.

Then $R_{A}$ has a fraction division ring: within the skew Laurent series ring for $A$, we show that the set of $rs^{-1}$ for $r \in R_{A}$ and nonzero $s \in R_{A}$ suffices. Since the skew Laurent series ring is in fact a division ring, it remains to show that this set is actually a subring. (It will then be a division ring because $rs^{-1}$ for nonzero $r, s \in R_{A}$ will have inverse $sr^{-1}$ in this set, noting that $r^{-1}$ is well-defined because the Laurent series ring is a division ring.) But using the right Ore condition, we can write any pair $rs^{-1}, tu^{-1}$ for $r, t \in R_{A}$ and nonzero $s, u \in R_{A}$ so that $s = u$, from which it follows that this set closed under addition, and we can also write any pair $rs^{-1}, tu^{-1}$ for $r \in R_{A}$ and nonzero $t, s, u \in R_{A}$ so that $s = t$, from which it follows that this set is closed under multiplication. Thus $R_{A}$ has a fraction division ring for finite $A$, and by the universal property it therefore has a fraction division ring for infinite $A$. Let $Q_{A}$ denote the fraction division ring of $R_{A}$.

Our Laurent series construction allows us to prove the following proposition:

\begin{prop}
    Let $C$ be a finite bipartite graph, and $A, B \subseteq C$ induced subgraphs with $C = A \cup B$. Then within $Q_{C}$, $Q_{A} \cap Q_{B} =Q_{A \cap B}$.
\end{prop}

\begin{proof}
    Let $v_{1}, \ldots, v_{n}$ be a fixed enumeration of the vertices of $C$; we may assume that $A \cap B $ forms an initial segment. Let $q \in Q_A \cap Q_B $; we show that $q \in Q_{A\cap B}$. Construct the skew Laurent series ring for $C$; within this ring, $q$ is in the skew Laurent series ring for $A\cap B$. It suffices to show that elements of $Q_{C}$, viewed as Laurent series in the $X_{i}$ for $v_{i} \notin A \cap B$ with coefficients in the skew Laurent series ring for $A \cap B$, even have those coefficients in $Q_{A \cap B}$; because $q$ is a constant term, this implies $q \in Q_{A \cap B}$. But the subring of the ring of  skew Laurent series for $C$, consisting of Laurent series in the $X_{i}$ for $v_{i} \notin A \cap B$ with coefficents in $Q_{A\cap B}$, is itself a division ring: this is the same as the above proof that the skew Laurent series ring for a bipartite graph is a division ring, replacing $\mathbb{Q}$ with $Q_{A\cap B}$. This division ring contains $R_{C}$, so by the universal property it contains $Q_C$.

\end{proof}

\begin{cor}\label{least support}

Let $A$ be a bipartite graph. Then every $q\in Q_{A}$ has a \emph{least support}: some finite $C \subseteq A $ such that $q\in Q_C$, and such that any $C' \subseteq A$ with $q\in Q_{C'}$ contains $C$. 

\end{cor}

\begin{proof}
    Note that if $q \in Q_{C}$, $q \in Q_{C_0}$ for some finite $C_0 \subseteq C$.  So it suffices to find some finite $C \subseteq A $ such that $q\in Q_C$, and such that any \textit{finite} $C' \subseteq A$ with $q\in Q_{C'}$ contains $C$. By the previous proposition, $q \in Q_{C}$, where $C$ is the intersection of all finite $C'$ such that $q\in C'$. Then $C$ is as desired.
\end{proof}

\subsection{Modules over division rings}\label{modules over division rings}

Let $Q$ be a division ring, and $V$ a, say, left $Q$-module. Then similarly to the case where $Q$ is a commutative field and $V$ a $Q$-vector space, every linearly independent set of $V$ can be extended to a basis. To see how this well-known fact follows just as in commutative field case, suppose $v_{1}, \ldots, v_{n}$ is linearly independent but $v_{1}, \ldots, v_{n}, v_{n+1}$ is dependent; we show that $v_{n+1}$ is in the span of $v_{1}, \ldots, v_{n}$. By linear independence of $v_{1}, \ldots, v_{n}$, the linear relation $q_{n+1}v_{n+1} = \sum^{n}_{i=1} q_{i}v_{i}$ for $q_{i} \in Q$ with some nonzero $q_{i}$ has $q_{n+1} \neq 0$. Dividing, $v_{n+1} \in \mathrm{span}( v_{1}, \ldots, v_{n})$, as desired. So every maximal linearly independent set spans, and therefore every linearly independent set can be extended to a basis. The same discussion applies to right modules over division rings.

We give equivalent conditions for a closure relation on modules over division rings, similar to the well-known case of vector spaces over algebraically closed fields:

\begin{prop}\label{strong inclusion}
    Let $Q \subseteq Q'$ be division rings with $Q$ a subring of $Q'$, and $M \subseteq M'$ with $M$ a left $Q$-module and $M'$ a left $Q'$-module, such that $M$ is a submodule of $M'$ with the induced $Q$-module structure. Then the following are equivalent:

    \begin{itemize}
        \item Every linearly independent set in $M$ over $Q$ remains linearly independent over $Q'$.

        \item Some basis for $M$ over $Q$ remains linearly independent over $Q'$.
    \end{itemize}
\end{prop}

\begin{proof}

($\Rightarrow$) Immediate from the observation that bases exist.

($\Leftarrow$) Suppose some basis $\langle m_{i}\rangle_{i \in I}$ for $M$ over $Q$ remains linearly independent over $Q'$. Viewing $Q'$ as a \textit{right} module over $Q$, construct the tensor product $Q' \otimes_{Q} M$, which for our purposes is the additive abelian group generated by symbols $q \otimes m$ for $q \in Q'$, $m \in M$ modulo the relations $(q + q') \otimes m - (q \otimes m + q' \otimes m), q\otimes (m+m') - (q \otimes m + q \otimes m'), qr \otimes m - q \otimes rm $ for $q, q' \in Q'$, $m, m' \in M$, $r \in Q$. Now choose a (right) basis $\langle q_{j}\rangle_{j \in J}$ for $Q'$ over $Q$. Because $\langle m_{i}\rangle_{i \in I}$ is linearly independent over $Q'$, every element of $\mathrm{span}_{Q'}(M)$'s representation (as a sum with finitely  many nonzero terms) of the form $\sum_{i \in I, j\in J}q_{j}r_{ij}m_{i}$ for some $r_{ij} \in Q$ is unique. (Otherwise, subtracting, there is some relation $\sum_{i \in I} e_{i} m_{i} =0 $ where $e_{i}$ is an expression of the form $\sum_{j \in J} q_{j}r'_{ij}$ for some $r'_{ij} \in Q$, and for some $i$ not all of the $r'_{ij}$ are zero. But then $Q' \ni e_{i} = \sum_{j \in J} q_{j}r'_{ij} \neq 0$. This contradicts linear independence of $\langle m_{i}\rangle_{i \in I}$ over $Q'$.) This implies that the scalar multiplication map $\mathrm{mult}: Q' \otimes_{Q} M \to M$ is injective: suppose $\mathrm{mult}(v) = \mathrm{mult}(w)$. Write $ v =\sum_{i \in I, j\in J}q_{j}\otimes r_{ij}m_{i}$ and $ w =\sum_{i \in I, j\in J}q_{j}\otimes r'_{ij}m_{i}$ for $r_{ij}, r'_{ij} \in Q$. Then $\sum_{i \in I, j\in J}q_{j}r_{ij}m_{i}=\mathrm{mult}(v) = \mathrm{mult}(w)=\sum_{i \in I, j\in J}q_{j}r'_{ij}m_{i}$. So $r_{ij}=r'_{ij}$ for all $i \in I, j \in J$, and $v =\sum_{i \in I, j\in J}q_{j}\otimes r_{ij}m_{i}= \sum_{i \in I, j\in J}q_{j}\otimes r'_{ij}m_{i} = w $.

Now let $\langle m'_{i}\rangle_{i \in I'}$ be a linearly independent subset of $M$ over $Q$; we show that $\langle m'_{i}\rangle_{i \in I'}$ remains linearly independent over $Q'$. Suppose $ I'' \subseteq I'$ is a finite set and $e_{i} \in Q'_{\neq 0}$ for $i \in I''$; what we want to show is that $\sum_{i \in I''} e_{i}m'_{i} \neq 0$. For each $i \in I''$, write $e_{i} = \sum_{j \in J}q_{j} r_{ij}$ with $r_{ij} \in Q$.  We show $\sum_{i \in I'', j \in J} q_{j}r_{ij}m'_{i} = \sum_{i \in I''} e_{i}m'_{i} \neq 0$. Since $\sum_{i \in I'', j \in J} q_{j}r_{ij}m'_{i} = \mathrm{mult}(\sum_{i \in I'', j \in J} q_{j} \otimes r_{ij} m'_{i})$ and $\mathrm{mult}$ is injective, it suffices to show $\sum_{i \in I'', j \in J} q_{j} \otimes r_{ij} m'_{i} \neq 0$. Choose any $i^{*} \in I''$, so $e_{i^{*}} = \sum_{j \in J} q_{j}r_{i^{*}j} \neq 0$; there is then some $j^{*} \in J$ such that $r_{i^{*}j^{*}} \neq 0$. There is a right $Q$-linear map $\ell_{Q'}: Q' \to Q$ such that $\ell_{Q'}(q_{j^{*}}) = 1$ and $\ell_{Q'}(q_{j}) = 0$ for $j \in J \backslash \{j^{*}\}$; similarly, there is a left $Q$-linear map $\ell_{M}: M \to Q$ such that $\ell_{M}(m'_{i^{*}}) = 1$ and $\ell_{Q'}(m'_{i}) = 0$ for $i \in I \backslash \{i^{*}\}$. Then let $\ell: Q'\otimes_{Q} M \to Q$ be the additive map with $\ell(q \otimes m)=\ell_{Q'}(q)\ell_{M}(m)$ for $q \in Q'$, $m \in M$. Then $\ell(\sum_{i \in I'', j \in J} q_{j} \otimes r_{ij} m'_{i})=r_{i^{*}j^{*}} \neq 0$, and $\sum_{i \in I'', j \in J} q_{j} \otimes r_{ij} m'_{i} \neq 0$, as desired.

\end{proof}

We use $(Q, M) \leq (Q', M')$ to denote $Q \subseteq Q'$, $M \subseteq M'$ satisfying the equivalent conditions of this proposition.

\section{The theory $T^{\neg\mathrm{StFk}}$}\label{Tnegstfk}

Our counterexample to the stable forking conjecture will be the theory of a module of infinite dimension over a noncommutative division ring, but the language is rather sensitive. The noncommutativity is at the heart of the instability of the forking relation, as will become clear. One can see with these requirements that the language must not interpret the ring operations, since by \cite{pillay1998supersimple} in the language of rings every division ring which has a supersimple theory is commutative. 

Let $\mathcal{L}$ be the language in sorts $\Gamma, V$ with

\begin{itemize}
\item Unary predicates $I, J$ on $\Gamma$.

    \item A binary function symbol $+: V \times V \to V$, and unary function symbols $q \cdot -: V \to V$ for $q \in \mathbb{Q}$

    \item A binary relation symbol $E$ on $\Gamma \times \Gamma$

    \item A binary function symbol $\cdot: \Gamma \times V \to V$.
\end{itemize}

Let $A = I_{A} \sqcup J_{A}$ be a finite bipartite graph with edge relation $E_{A}$ as in the previous section. We associate any $Q_{A}$-module $M$ with an $\mathcal{L}$-structure as follows: interpret the sort $\Gamma$ as $A$, and the sort $V$ as $M$. Interpret $+$ as addition on $M$ and $q \cdot -$ for $q \in \mathbb{Q}$ as scalar multiplication on $M$ by elements of $\mathbb{Q}$, $I$ and $J$ as $I_{A}$ and $J_{A}$, and $E$ as the graph relation $E_{A}$ on $A$. Finally, interpret $\cdot$ as the action $(v, m) \mapsto X_{v}m$.

Then we can recover the $Q_{A}$-module structure on $M$ from this $\mathcal{L}$-structure canonically: by the definition of $R_{A}$ and of the division ring of fractions, any element of $Q_{A}$ will be of the form $p(X_{v_{1}}, \ldots, X_{v_{n}})(q(X_{w_{1}}, \ldots, X_{w_{k}}))^{-1}$ where $v_{i}, w_{i} \in A$ and $p(\xi_{1}, \ldots \xi_{n}), q(\zeta_{1}, \ldots \zeta_{k})$ are polynomial expressions in the variables $\xi_{i}$ and in the variables $\zeta_{i}$ (say, $\mathbb{Q}$-linear combinations of monomials in the $\xi_{i}$ and in the $\zeta_{i}$, where those monomials are arbitrary expressions of the form $\xi^{k_{1}}_{i_{1}} \ldots \xi^{k_{\ell}}_{i_{\ell}}$ for $k_{i} \in \mathbb{Z}_{\geq 0}$.) Then addition of scalars in $Q_{A}$ corresponds to addition of their scalar multiplication maps, and similarly for scalar multiplication of scalars in $Q_{A}$ by elements of $\mathbb{Q}$; multiplication of scalars in $Q_{A}$ corresponds to composition of their scalar multiplication maps; and multiplicative inversion of scalars in $Q_{A}$ corresponds to inversion of their scalar multiplication maps. So there is an $\mathcal{L}$-formula $\varphi(x_{1}, \ldots x_{n}; z_{1}, \ldots z_{k}; u, u')$ depending only on $p(\xi_{1}, \ldots \xi_{n}), q(\zeta_{1}, \ldots \zeta_{k})$ such that $p(X_{v_{1}}, \ldots, X_{v_{n}})(q(X_{w_{1}}, \ldots, X_{w_{k}}))^{-1}m=m'$ exactly when $\models \varphi(v_{1}, \ldots v_{n}; w_{1}, \ldots w_{k}; m, m')$ for $v_{1}, \ldots v_{n}, w_{1}, \ldots w_{k}\in \Gamma$, $m, m' \in V$. We thus recover the $Q_{A}$-module structure on $M$ from this $\mathcal{L}$-structure. This justifies the following terminology: when we speak of a $Q_{A}$-module $M$ considered in the language $\mathcal{L}$, we mean this $\mathcal{L}$-structure with domain consisting of sorts $A$ and $M$. This is despite that, literally speaking, these sorts $A$ and $M$ are not the same pair of sorts as \textit{the scalar field $Q_{A}$} and the module $M$.

Let $A \subseteq B$ be bipartite graphs with $A$ an induced subgraph of $B$, and $M \subseteq N$ with $M$ a $Q_{A}$-module, $N$ a $Q_{B}$-module, and $M$ a submodule of $N$ considered as a $Q_{A}$-module. Then let $(A, M) \subseteq (B, N)$ denote the inclusion of $\mathcal{L}$-structures where $(A, M)$ is the $Q_{A}$-module $M$ considered as an $\mathcal{L}$-structure and $(B, N)$ is the $Q_{B}$-module $N$ considered as an $\mathcal{L}$-structure. We let $(A, M) \leq (B, N)$ denote $(Q_{A}, M) \leq (Q_{B}, N)$ as in the previous section. (Remember that this is the assertion that $Q_{A}$-independent sets in $M$ remain $Q_{B}$-independent, or equivalently that some $Q_{A}$-basis for $M$ remains $Q_{B}$-independent, as in Proposition \ref{strong inclusion}.)

Let $G$ be the random graph. Then there is a $Q_{G}$-module $M$ of infinite dimension (i.e., no finite set spans $M$ over $Q_{G}$). Let $T^{\neg \mathrm{StFk}}$ be the theory of this structure in the language $\mathcal{L}$. We show that $T^{\neg \mathrm{StFk}}$ is a counterexample to the stable forking conjecture.

The $\mathcal{L}$-structure $(G, M)$ has the following three properties:

\begin{itemize}
    \item The graph $G$ is a model of the theory of the random bipartite graph.

    \item The structure $(G, M)$ is a $Q_{G}$-module considered in the language $\mathcal{L}$.

    \item The $Q_{G}$-module $M$ is infinite-dimensional (no finite set spans).

\end{itemize}

We first show that the first two of these properties of $\mathcal{L}$-structures are expressible by a set of first-order sentences, so every model of $T^{\neg \mathrm{StFk}}$ has these properties.

That $G$ is a model of the theory of the random bipartite graph is immediately expressible by a set of first-order sentences. To show that the second condition is expressible, let us first show that we can express the property of the structure $(G, M)$ being an $R_{G}$-module considered in the language $\mathcal{L}$, where we recover the $R_{G}$-module substructure just as we saw how to recover the $Q_{G}$-module structure above. Given some vertices $v_{1}, \ldots, v_{n} \in G$, we may uniformly determine in their quantifier-free $I, J, E$ type which pairs of polynomial expressions in $X_{v_{1}}, \ldots, X_{v_{n}}$ represent the same element of $R_{G}$. Then the following condition is expressible, and equivalent to being an $R_{G}$-module considered in the language $\mathcal{L}$: first, $M$ is a $\mathbb{Q}$-vector space, where the vertices of $G$ represent actions by individual $\mathbb{Q}$-linear maps; thus polynomial expressions determine actions by individual $\mathbb{Q}$-linear maps from those actions by addition, rational scalar multiplication and composition, as above. And second, when the quantifier-free $I, J, E$-type of $v_{1}, \ldots, v_{n} \in G$ is $q(x_{1}, \ldots, x_{n})$, then for every pair of polynomial expressions in $x_{1}, \ldots x_{n}$ from the predetermined list of pairs of polynomial expressions that evaluate to the same element of $R_{G}$ on $X_{w_{1}}, \ldots X_{w_{n}}$ for $w_{1}, \ldots, w_{n} \in G$ realizing $q(x_{1}, \ldots, x_{n})$, the actions on $M$ determined by those polynomial expressions from the actions represented by $v_{1}, \ldots, v_{n}$ are identical. We now show that the full condition of being a $Q_{G}$-module (and not just an $R_{G}$-module) considered in the language $\mathcal{L}$ is expressible. Similarly to the case of polynomial expressions, given some vertices $v_{1}, \ldots, v_{n} \in G$, we may uniformly determine in their quantifier-free $I, J, E$ type which (right) \textit{fractions} of polynomial expressions in $X_{v_{1}}, \ldots, X_{v_{n}}$ represent the same element of $Q_{G}$. Then the following condition is expressible, and equivalent to being an $Q_{G}$-module considered in the language $\mathcal{L}$: first, $M$ is an $R_{G}$-module considered in the language $\mathcal{L}$ where scalar multiplication is always invertible; thus \textit{fractions} of polynomial expressions determine actions by individual $\mathbb{Q}$-linear maps from those actions by addition, rational scalar multiplication, composition, and \textit{inversion}, as above. And second, when the quantifier-free $I, J, E$-type of $v_{1}, \ldots, v_{n} \in G$ is $q(x_{1}, \ldots, x_{n})$, then for every pair of \textit{fractions} of polynomial expressions in $x_{1}, \ldots x_{n}$ from the predetermined list of pairs of polynomial expressions that evaluate to the same element of $R_{G}$ on $X_{w_{1}}, \ldots X_{w_{n}}$ for $w_{1}, \ldots, w_{n} \in G$ realizing $q(x_{1}, \ldots, x_{n})$, the actions on $M$ determined by those fractions of polynomial expressions from the actions represented by $v_{1}, \ldots, v_{n}$ are identical.

We next show that there is a set of first-order sentences that is satisfied by every $\mathcal{L}$-structure satisfying all three properties, and such that every \textit{$\aleph_{1}$-saturated} $\mathcal{L}$-structure satisfying this set of first-order sentences satisfies all three properties. Thus every \textit{$\aleph_{1}$-saturated} model of $T^{\neg\mathrm{StFk}}$ will satisfy these properties. The following property of $(G, M)$ is expressible by a set of first-order sentences, and by $\aleph_{1}$-saturatedness, will be as desired: first, $(G, M)$ satisfies the first two properties, $G$ being a model of the theory of the random graph and $(G, M)$ being a $Q_{G}$-module considered in the language $\mathcal{L}$. (We just showed that this is expressible by a set of first-order sentences.) And second, given $n, m, \ell < \omega$ and fractions of polynomial expressions $\Phi^{j}_{i}(x_{1}, \ldots x_{m} )$ for $1 \leq i \leq n$, $j \leq \ell$, for all $m_{1}, \ldots m_{n} \in M$ and $v^{1}_{1}, \ldots, v^{1}_{m}, \ldots, v^{\ell}_{1}, \ldots, v^{\ell}_{m}  \in G$, there is some $m_{n+1} \in M$ with $m_{n+1} \neq \sum^{n}_{i =1} \Phi^{j}_{i}(X_{v^{j}_{1}}, \ldots, X_{v^{j}_{m}}) m_{i}$ for each $j \leq \ell$. (This is expressible by a set of first-order sentences by the above discussion of how the $\mathcal{L}$-structure canonically determines the $Q_{G}$-module structure.)

Let $\mathbb{M} \models T^{\neg \mathrm{StFk}}$ be a sufficiently saturated ambient model of $T^{\neg \mathrm{StFk}}$--say, $\kappa=\lambda^{+}$-saturated, for $\lambda$ large enough to apply the Erdős–Rado theorem, so particularly $\kappa \geq \aleph_{1}$. Then $\mathbb{M}$ satisfies all three of the above properties. In fact, by the same reasoning by which every $\aleph_{1}$-saturated model of $T^{\neg \mathrm{StFk}}$ has infinite dimension, $\mathbb{M}$ has dimension $\kappa$. We show that the strong inclusion relation $\leq$ defined above yields an embedding property for the ambient or ``monster" model $\mathbb{M}$.

\begin{lemma}\label{embedding property for strong inclusions}
    Let $A, B$ be small (i.e. of size less than $\kappa$) $\mathcal{L}$-structures with $A \leq B$, $A \leq \mathbb{M}$. (So specifically, $A$ is a $Q_{\Gamma(A)}$-module, and $B$ is a $Q_{\Gamma(B)}$-module.) Then there is an embedding $\iota: B \hookrightarrow \mathbb{M}$ of $\mathcal{L}$-structures with $\iota|_{A}=\mathrm{id}_{A}$ and $\iota(B) \leq \mathbb{M}$.
\end{lemma}

\begin{proof}
    Every inclusion of the form $(C, M) \leq (D, N)$ can be decomposed as $(C, M) \subseteq (D, \mathrm{span}_{D}(M))\subseteq (D, N)$. Then $(C, M) \leq (D, \mathrm{span}_{Q_{D}}(M))\leq (D, N)$, the first of these strong inclusions following from $(C, M) \leq (D, N)$, the second immediate from the definition of $\leq$. The first inclusion is of the form $A \leq B$ where $V(B) = \mathrm{span}_{Q_{\Gamma(B)}}(V(A))$, and the second is an inclusion of the form $A \leq B$ where $\Gamma(A) = \Gamma(B)$. The relation $\leq$ is transitive, so it suffices to prove the statement of this lemma assuming each of these two cases.

    Let us first consider inclusions of the form $A \leq B$ where $V(B) = \mathrm{span}_{Q_{\Gamma(B)}}(V(A))$. Let $\langle v_{i} \rangle_{i \in \mathcal{I}}$ be a $Q_{\Gamma(A)}$-basis for $V(A)$; then by $A \leq B$, $\langle v_{i} \rangle_{i \in \mathcal{I}}$ is linearly independent over $Q_{\Gamma(B)}$, so, by $V(B) = \mathrm{span}_{Q_{\Gamma(B)}}(V(A))$, is a $Q_{\Gamma(B)}$-basis for $V(B)$. Since $\Gamma(\mathbb{M})$ is a $\kappa$-saturated model of the theory of the random graph, there is an induced embedding $\iota: \Gamma(B) \hookrightarrow \Gamma(\mathbb{M}) $ with $\iota|_{\Gamma(A)}=\mathrm{id}_{\Gamma(A)}$. There is then a unique ring embedding $\tilde{\iota}: Q_{\Gamma(B)} \hookrightarrow Q_{\Gamma(\mathbb{M})}$ with $X_{v} \mapsto X_{\iota(v)}$ for $v \in \Gamma(B)$. For $q_{i} \in Q_{\Gamma(B)}$, extend $\iota: \Gamma(B) \hookrightarrow \Gamma(\mathbb{M})$ to an embedding  $\iota: B \hookrightarrow \mathbb{M}$ with $\iota|_{A}=\mathrm{id}_{A}$ by $\iota(\sum_{i \in \mathcal{I}} q_{i}v_{i})=\sum_{i \in \mathcal{I}} \tilde{\iota}(q_{i})v_{i} $ for $q_{i} \in Q_{\Gamma(B)}$ (with all but finitely many of the $q_{i}$ equal to $0$). Then $\langle v_{i} \rangle_{i \in \mathcal{I}}$ spans $V(\iota(B))$ over $\tilde{\iota}(Q_{\Gamma(B)})= Q_{\iota(\Gamma(B))}$. By $A \leq \mathbb{M}$ and the fact that $\langle v_{i} \rangle_{i \in \mathcal{I}}$ is linearly independent over $Q_{\Gamma(A)}$, $\langle v_{i} \rangle_{i \in \mathcal{I}}$ is linearly independent over $Q_{\Gamma(\mathbb{M})}$, so particularly over $Q_{\iota(\Gamma(B))}$; thus $\langle v_{i} \rangle_{i \in \mathcal{I}}$ is a $Q_{\iota(\Gamma(B))}$-basis for $V(\iota(B))$. So a $Q_{\iota(\Gamma(B))}$-basis for $V(\iota(B))$ remains linearly independent over $Q_{\Gamma(\mathbb{M})}$, and $\iota(B) \leq \mathbb{M}$, as desired.

    The remaining case is where $A \leq B$ where $\Gamma(A) = \Gamma(B)$. Let $\langle v_{i} \rangle_{i \in \mathcal{I}}$ be a $Q_{\Gamma(A)}=Q_{\Gamma(B)}$-basis for $V(A)$, and complete this to a $Q_{\Gamma(A)}=Q_{\Gamma(B)}$-basis $\langle v_{i} \rangle_{i \in \mathcal{J}}$ for $V(B)$, where $\mathcal{I} \subseteq \mathcal{J}$. Define $\iota: \Gamma(A)=\Gamma(B) \hookrightarrow \mathbb{M}$ to be the identity. By $A\leq \mathbb{M}$, $\langle v_{i} \rangle_{i \in \mathcal{I}}$, being linearly independent over $Q_{\Gamma(A)}$, is linearly independent over $Q_{\Gamma(\mathbb{M})}$. Since $V(\mathbb{M})$ is $\kappa$-dimensional over $Q_{\Gamma(\mathbb{M})}$ (i.e., no set of size less than $\kappa$ spans), $\langle v_{i} \rangle_{i \in \mathcal{I}}$ can be extended to a $Q_{\Gamma(\mathbb{M})}$-linearly independent set $\langle w_{i} \rangle_{i \in \mathcal{J}}$ with $v_{i} = w_{i}$ for $i \in \mathcal{I}$. For $q_{i} \in Q_{\Gamma(A)}=Q_{\Gamma(B)}$, extend $\Gamma(A)=\Gamma(B)\hookrightarrow \Gamma(\mathbb{M})$ to a map  $\iota: B \hookrightarrow \mathbb{M}$ with $\iota|_{A}=\mathrm{id}_{A}$ by $\iota(\sum_{i \in \mathcal{J}} q_{i}v_{i})=\sum_{i \in \mathcal{J}} q_{i}w_{i}$. Then $\langle w_{i} \rangle_{i \in \mathcal{J}}$ spans $V(\iota(B))$ over $Q_{\Gamma(A)}=Q_{\iota(\Gamma(B))}$. Because  $\langle w_{i} \rangle_{i \in \mathcal{J}}$ are linearly independent over $Q_{\Gamma(\mathbb{M})}$, they are in particular linearly independent over $Q_{\Gamma(\iota(B))}$, so a basis for $V(\iota(B))$ over $Q_{\Gamma(\iota(B))}$. So a basis for $V(\iota(B))$ over $Q_{\Gamma(\iota(B))}$ remains linearly independent over $Q_{\Gamma(\mathbb{M})}$, and again $\iota(B) \leq \mathbb{M}$ as desired.
\end{proof}

We now show that the strong inclusion relation $\leq $ yields a closure relation on (small) subsets of $\mathbb{M}$.

\begin{prop}\label{closure with respect to strong inclusion}
    For small $A \subset \mathbb{M}$, there is a small set $\mathrm{cl}(A) \supset A$ contained in $\mathbb{M}$ with $\mathrm{cl}(A) \leq \mathbb{M}$, and with $\mathrm{cl}(A)$ contained in every $B \supset A$ with $B \leq \mathbb{M}$
\end{prop}

\begin{proof}

Call a (small) set $A \subset \mathbb{M}$ \textit{closed} if it satisfies the following conditions:

\begin{itemize}
    \item The set $V(A)$ is a submodule of $V(\mathbb{M})$ considered as a $Q_{\Gamma(A)}$-module.

    \item For $\langle v_{i}\rangle^{n}_{i = 1}$ linearly independent over $Q_{\Gamma(\mathbb{M})}$ with $v_{i} \in V(A)$, and $v_{n+1} \in V(A)$, suppose $v_{n+1} = \sum^{n}_{i=1} q_{i}v_{i}$ for $q_{i} \in Q_{\Gamma(\mathbb{M})}$; thus by linear independence of $\langle v_{i}\rangle^{n}_{i = 1}$
    over $Q_{\Gamma(\mathbb{M})}$, these $q_{i}$ are the unique $q'_{i} \in Q_{\Gamma(\mathbb{M})}$ with $v_{n+1} = \sum^{n}_{i=1} q'_{i}v_{i}$.\footnote{These equations imply $\sum^{n}_{i=1} q'_{i}v_{i}=\sum^{n}_{i=1} q_{i}v_{i}$, so $\sum^{n}_{i=1} (q_{i}-q'_{i})v_{i}=0$, so $q_{i}-q'_{i} = 0$ for $1 \leq i \leq n$.} Then for $1 \leq i \leq n$, the least support (as in Corollary \ref{least support}) of $q_{i}$ is contained in $\Gamma(A)$.
\end{itemize}

(This is similar to closing off under the coordinate functions found in, say, the algebraic closure within the theory of infinite-dimensional vector spaces over an algebraically closed field, but now each coordinate function is represented by its value's least support.)

For each $A \subset \mathbb{M}$, there is a closed set $\mathrm{cl}(A) \supset A$ contained in every closed set containing $A$. (By the uniqueness remark in the second condition, and finiteness of the least support, we can find this by closing off under multivalued partial functions with finitely many values at each argument.) So it suffices to show that $A \subset \mathbb{M}$ is closed if and only if $A \leq \mathbb{M}$. 

 First, suppose $A \leq \mathbb{M}$; we show that $A$ is closed. The first condition is just the assumption required for the strong inclusion $\leq$. Now suppose  $\langle v_{i}\rangle^{n}_{i = 1}$ are linearly independent over $Q_{\Gamma(\mathbb{M})}$ with $v_{i} \in V(A)$, and $v_{n+1} \in V(A)$ with $v_{n+1} \in \sum^{n}_{i=1} q_{i}v_{i}$ for $q_{i} \in Q_{\Gamma(\mathbb{M})}$. Then $\langle v_{i}\rangle^{n+1}_{i = 1}$ are linearly dependent over $Q_{\Gamma(\mathbb{M})}$, so by $A \leq \mathbb{M}$, $\langle v_{i}\rangle^{n+1}_{i = 1}$ are linearly dependent over $Q_{\Gamma(A)}$. There is then a linear relation $\sum^{n+1}_{i = 1}s_{i}v_{i}$, with $s_{i} \in Q_{\Gamma(A)}$ and not all $s_{i}$ equal to zero. But $\langle v_{i}\rangle^{n}_{i = 1}$ are linearly independent over $Q_{\Gamma(\mathbb{M})}$, so in particular are linearly independent over $Q_{\Gamma(A)}$. Thus $s_{n+1} \neq 0$, yielding a relation $v_{n+1} = \sum^{n}_{i=1} q'_{i}v_{i}$ for $q'_{i} \in Q_{\Gamma(A)}$; by the uniqueness remarked on in the statement of the second condition, $q_{i} = q'_{i} \in  Q_{\Gamma(A)} $ for each $1 \leq i \leq n$. But then the least support of $q_{i}$ must be contained in $\Gamma(A)$, as desired for the second condition.

Now suppose that $A \subset \mathbb{M}$ is closed; we show that $A \leq \mathbb{M}$. By the first condition, $A$ is a candidate for the condition that $A \leq \mathbb{M}$. What we want to show is that, for $\langle v_{i}\rangle_{i \in \mathcal{I}}$ a $Q_{\Gamma(\mathbb{M})}$-linearly dependent set with $v_{i} \in V(A)$, $\langle v_{i}\rangle_{i \in \mathcal{I}}$ is linearly dependent over $Q_{\Gamma(A)}$. The set $\langle v_{i}\rangle_{i \in \mathcal{I}}$ contains a finite \textit{minimal} $Q_{\Gamma(\mathbb{M})}$-linearly dependent set; since if some $v_{i}$ is equal to $0$ we are done, we can assume this is a set $\langle v_{i_{k}}\rangle^{n+1}_{k =1}$ of size at least $2$. Let $\sum^{n+1}_{k = 1}s_{k}v_{i_{k}}$ be a linear relation with $s_{k} \in Q_{\Gamma(\mathbb{M})}$ and not all $s_{k}$ equal to zero; then by minimality, $s_{n+1} \neq 0$. Subtracting and dividing, there is a linear relation of the form $v_{i_{n+1}} = \sum^{n}_{k=1} q_{i_{k}}v_{i_{k}}$ with $q_{k} \in Q_{\Gamma(\mathbb{M})}$. By the second condition, for $1 \leq k \leq n$, $q_{k}$ has least support in $\Gamma(A)$, so $q_{k} \in Q_{\Gamma(A)}$. Thus $\langle v_{i_{k}}\rangle^{n+1}_{k =1}$ is linearly dependent over $Q_{\Gamma(A)}$, as desired.

\end{proof}

This gives us a description of types in $T^{\neg \mathrm{StFk}}$.

\begin{cor}\label{quantifier elimination}

Let $\mathbb{M}$, $\mathbb{M}'$ be $\kappa$-saturated models of $T^{\neg \mathrm{StFk}}$, and let $A \subset \mathbb{M}$, $A' \subset \mathbb{M}'$ be small sets. Then $A$ and $A'$ satisfy the same type if and only if there is an isomorphism $\iota: \mathrm{cl}(A) \to \mathrm{cl}(A')$ with $\iota(A) = A'$ (as enumerated sets.)

\end{cor}

\begin{proof}
    First, suppose $A$ and $A'$ satisfy the same type; then there is an isomorphism $\iota: B \to B'$ with $B \supseteq \mathrm{cl}(A)$, $B' \supseteq \mathrm{cl}(A')$,  $\iota(B) = B'$ (as enumerated sets) and $\mathrm{tp}(B)=\mathrm{tp}(B')$. Then $A' \subseteq \iota(\mathrm{cl}(A)) \leq \mathbb{M}'$, so $\mathrm{cl}(A') \subseteq \iota(\mathrm{cl}(A))$. Similarly, $\mathrm{cl}(A) \subseteq \iota^{-1}(\mathrm{cl}(A'))$, so $\mathrm{cl}(A') \supseteq \iota(\mathrm{cl}(A))$, and $\iota(\mathrm{cl}(A))=\mathrm{cl}(A')$. Restricting to $\mathrm{cl}(A)$ gives us our isomorphism $\iota: \mathrm{cl}(A) \to \mathrm{cl}(A')$ with $\iota(A) = A'$.

    The converse follows from the fact that isomorphisms $\iota: A \to A'$ with $A \leq \mathbb{M}$, $A' \leq \mathbb{M}'$ form a back-and-forth system, by the previous proposition and lemma. (If there is an isomorphism between $B \leq \mathbb{M}$, $B' \leq \mathbb{M}$ (with $\iota(B) = B'$ as enumerated sets), that these isomorphisms form a back-and-forth system implies that $B$ and $B'$ satisfy the same type. Apply this to $\mathrm{cl}(A)$ and $\mathrm{cl}(A')$ (with compatible orderings chosen on $\mathrm{cl}(A)$ and $\mathrm{cl}(A')$ extending those on $A$ and $A'$).)
\end{proof}

\section{Falseness of the stable forking conjecture}\label{counterexample verification}

\subsection{Independence in and simplicity of $T^{\neg\mathrm{StFk}}$}\label{simplicity of the example}

Our next goal is to show that $T^{\neg\mathrm{StFk}}$ is simple, and thus a candidate for being a counterexample to the stable forking conjecture, as well as to characterize forking in $T^{\neg\mathrm{StFk}}$, toward showing that the conclusion of the stable forking conjecture fails there. We will employ a standard method from the literature for showing that a theory is simple, and characterizing forking in that theory: use the fact that simplicity is characterized by the existence of an abstract relation between sets satisfying certain desiderata, and that such a relation, if it exists, coincides with the independence relation given by non-forking. The original version of this characterization, which we will use, is the following classical theorem of Kim and Pillay:

\begin{fact}\label{Kim-Pillay theorem}(Kim and Pillay, \cite{KP99}, Theorem 4.2)

Let $\mathbb{M} \models T$ be a $\kappa$-saturated model of a theory $T$. Suppose there is a relation $\ind$, with $A \ind_{C} B$ defined for small $A, B, C \subset \mathbb{M}$, satisfying the following properties:

\begin{itemize}
    \item \emph{Invariance:} If $A \ind_{C} B $ and $ABC \equiv A'B'C'$, $A' \ind_{C'} B'$.\footnote{If $\mathbb{M}$ is strongly $\kappa$-homogeneous, so any two small sets satisfying the same type are $\mathrm{Aut}(\mathbb{M})$-conjugates of each other, then this is the same as invariance of $\ind$ under automorphisms of $\mathbb{M}$. To avoid the set-theoretic assumptions that the existence of strongly $\kappa$-homogeneous models requires, we use this alternative statement that whether $A \ind_{B} C$ depends only on the type of $ABC$.}
    
    \item \emph{Local character}: For any finite tuple $a$ and small set $C$, there is a set $C_{0} \subset C$ with $|C_{0}| \leq |T|$ such that $a\ind_{C_0} C$.

    \item \emph{Finite character:} If $A \nind_{C} B $, there is some finite $b \subset B$ such that $A \nind_{C} b$.

    \item \emph{Monotonicity:} If $A \ind_{C} B$ and $B' \subseteq B$, then $A \ind_{C} B'$.\footnote{This is originally incorporated by Kim and Pillay into the finite character assumption; see, e.g., \cite{Adl07} for the terminology.}

    \item \emph{Extension:} For any small sets $A, B, C$, there is $A' \equiv_{C} A$ with $A' \ind_{C} B$.\footnote{This condition is also referred to as ``full existence"; see, e.g., \cite{D21}. In Kim and Pillay's statement of this property, as well as transitivity and base monotonicity, $A$ is assumed to be finite. Since there is no harm in strengthening any of the conditions for $\ind$, we drop this assumption.}
    \item \emph{Symmetry}: if $A \ind_{C} B$, then $B \ind_{C} A$.
    \item \emph{Transitivity}: if $A \ind_{D} C$, $A \ind_{C} B$ and $D \subset C \subset B $, then $A \ind_{D} B$.

    \item \emph{Base monotonicity}: if $A \ind_D B$ and $D \subset C \subset B $, then $A \ind_C B$.\footnote{This is originally incorporated by Kim and Pillay into the transitivity assumption; see, e.g., \cite{KR17} for the terminology.} 

    \item \emph{Independence theorem over models}: If $M \prec \mathbb{M}$, $A_{1} \ind_{M} B$, $A_{2} \ind_{M} C$, $C \ind_{M} B$, and $A_{1} \equiv_{M} A_{2}$, then there is some $A \ind_{M} BC$ with $A \equiv_{MB} A_{1}$ and $A \equiv_{MC} A_{2}$.

    Then $T$ is simple and $\ind$ coincides with the independence relation given by non-forking.

\end{itemize}

\end{fact}

(See \cite{KR17} for a version of this for $\mathrm{NSOP}_{1}$ theories.) To show $T^{\neg\mathrm{StFk}}$ is simple and to characterize forking, let us define the following independence relation, which we will show satisfies the hypotheses of the Kim-Pillay theorem:

\begin{itemize}
    \item For $C \subseteq A, B$ with $A= \mathrm{cl}(A)$, $B= \mathrm{cl}(B)$ and $C= \mathrm{cl}(C)$, $A \ind_{C} B$ if $\Gamma(A) \cap \Gamma(B) = \Gamma(C)$ and $\mathrm{span}_{Q_{\Gamma(\mathbb{M})}}(V(A))\cap\mathrm{span}_{Q_{\Gamma(\mathbb{M})}}(V(B)) = \mathrm{span}_{Q_{\Gamma(\mathbb{M})}}(V(C)) $. 

    \item For any $C \subseteq A, B$, $A \ind_{C} B$ if $\mathrm{cl}(AC) \ind_{\mathrm{cl}(C)} \mathrm{cl}(BC)$.
\end{itemize}

Let us now show the hypotheses of the Kim-Pillay theorem for $\ind$.

\textit{Invariance:} Suppose that if $A \ind_{C} B $ and $ABC \equiv A'B'C'$. We can then find a map $\iota: D \to D'$ for $D \supset \mathrm{cl}(A)\mathrm{cl}(B)\mathrm{cl}(C) $ and $D' \supset \mathrm{cl}(A')\mathrm{cl}(B')\mathrm{cl}(C') $ such that $\iota(ABC) = A'B'C'$, $\iota(D) = D'$ and $\mathrm{tp}(D) = \mathrm{tp}(D')$. As in the first paragraph of the proof of Corollary \ref{quantifier elimination}, $\iota(\mathrm{cl}(A))=\mathrm{cl}(A')$, $\iota(\mathrm{cl}(B))=\mathrm{cl}(B')$ and  $\iota(\mathrm{cl}(C))=\mathrm{cl}(C')$, so $\mathrm{cl}(A)\mathrm{cl}(B)\mathrm{cl}(C)$ and $\mathrm{cl}(A')\mathrm{cl}(B')\mathrm{cl}(C')$ satisfy the same type. But then, $\Gamma(\mathrm{cl}(A)) \cap \Gamma(\mathrm{cl}(B)) = \Gamma(\mathrm{cl}(C))$ and $\mathrm{span}_{Q_{\Gamma(\mathbb{M})}}(V(\mathrm{cl}(A)))\cap\mathrm{span}_{Q_{\Gamma(\mathbb{M})}}(V(\mathrm{cl}(B))) = \mathrm{span}_{Q_{\Gamma(\mathbb{M})}}(V(\mathrm{cl}(C))) $ imply $\Gamma(\mathrm{cl}(A')) \cap \Gamma(\mathrm{cl}(B')) = \Gamma(\mathrm{cl}(C'))$ and $\mathrm{span}_{Q_{\Gamma(\mathbb{M})}}(V(\mathrm{cl}(A')))\cap\mathrm{span}_{Q_{\Gamma(\mathbb{M})}}(V(\mathrm{cl}(B'))) = \mathrm{span}_{Q_{\Gamma(\mathbb{M})}}(V(\mathrm{cl}(C'))) $. So $A' \ind_{C'} B'$, as desired.

\textit{Local character:} 

To show local character, we first prove the following about $\mathrm{cl}(a)$ for a set $a$.

\begin{claim}
    For any set $a$, $V(\mathrm{cl}(a))=\mathrm{span}_{\Gamma(\mathrm{cl}(a))}(V(a))$. Moreover, if $a$ is finite, then $\Gamma(\mathrm{cl}(a))$ is finite. 
\end{claim}

\begin{proof}
    Let $a' := \mathrm{cl}(a)$. First, $V(a') = \mathrm{span}_{Q_{\Gamma(a')}}(V(a))$: $a \subset (\Gamma(a'), \mathrm{span}_{Q_{\Gamma(a')}}(V(a))) \leq (\Gamma(a'), V(a'))$ by the definition of $\leq$, so $a'= \mathrm{cl}(a)=(\Gamma(a'), \mathrm{span}_{Q_{\Gamma(a')}}(V(a)))$. Second, if $a$ is finite, then $\Gamma(a')$ is finite: let $a^{V}_{0} \subseteq V(a)$ be a maximal $Q_{\Gamma(a')}$-linearly independent subset of $V(a)$. Then by the argument at the beginning of Subsection \ref{modules over division rings}, $V(a) \subseteq \mathrm{span}_{Q_{\Gamma(a')}}(a^{V}_{0})$. Since $V(a)$ is finite, there is then a finite subset $a^{\Gamma}_{0} \subset \Gamma(a')$ with $a_{0}^{\Gamma} \supset \Gamma(a)$ such that $V(a) \subseteq \mathrm{span}_{Q_{a_{0}^{\Gamma}}}(a^{V}_{0})$. Then $a_{0}^{V}$ is $Q_{\Gamma(a')}$-linearly independent, and thus $Q_{a_{0}^{\Gamma}}$-linearly independent, and therefore a $Q_{a_{0}^{\Gamma}}$-basis for $\mathrm{span}_{Q_{a_{0}^{\Gamma}}}(a^{V}_{0})$. So $(\Gamma(a), V(a)) \subset (a^{\Gamma}_{0}, \mathrm{span}_{Q_{a_{0}^{\Gamma}}}(a^{V}_{0}))\leq (\Gamma(a'), V(a'))$, so $a'= \mathrm{cl}(a) = (a^{\Gamma}_{0}, \mathrm{span}_{Q_{a_{0}^{\Gamma}}}(a^{V}_{0}))$ and $\Gamma(a')=a^{\Gamma}_{0}$ is finite.
\end{proof}

Now, suppose $a$ is a finite set; we show that for any set $C$, there is a \textit{finite} set $C_{0} \subset C$ such that $a \ind_{C_{0}} C$. (This is stronger than what is needed for local character, where $C_{0}$ is only required to be countable.)

First, because $V(\mathrm{cl}(a))=\mathrm{span}_{\Gamma(\mathrm{cl}(a))}(V(a))$, $\mathrm{span}_{\Gamma(\mathbb{M})}(\mathrm{cl}(a))=\mathrm{span}_{\Gamma(\mathbb{M})}(a)$. So the set $\mathrm{span}_{\Gamma(\mathbb{M})}(\mathrm{cl}(a)) $ will be spanned over $\Gamma(\mathbb{M})$ by a finite set. Therefore, a basis for $\mathrm{span}_{\Gamma(\mathbb{M})}(\mathrm{cl}(a)) \cap \mathrm{span}_{\Gamma(\mathbb{M})}(\mathrm{cl}(C))$ will be finite by the following claim, about modules over division rings, generalizing facts about dimension from linear algebra over commutative fields:

\begin{claim}
    Let $Q$ be a division ring, and $M$ a (left) $Q$-module which is the span of a finite set. Then $M$ has no infinite linearly independent set.
\end{claim}

\begin{proof}
    It suffices to show the following claim: for any $m_{1}, \ldots, m_{n} \in M$, any $m'_{1} \ldots, m'_{n+1} \in \mathrm{span}(m_{1}, \ldots m_{n})$ are linearly dependent. We will show this by induction on $n$. We may assume that $m'_{n+1}$ is nonzero, so we can assume without loss of generality that $m'_{n+1} = \sum^{n}_{i=1} q_{i}m_{i}$ for $q_{i} \in Q$ and $q_{n} \neq 0$. Then there are $s_{i} \in Q$ for $1 \leq i \leq n$ such that $m'_{i}-s_{i}m'_{n+1} \in \mathrm{span}(m_{1}, \ldots m_{n-1})$. By induction, these $n$ elements $m'_{i}-s_{i}m'_{n+1}$ are linearly dependent, from which it follows that  $m'_{1} \ldots, m'_{n+1}$ are linearly dependent.
\end{proof}

Since $\mathrm{span}_{\Gamma(\mathbb{M})}(\mathrm{cl}(a)) \cap \mathrm{span}_{\Gamma(\mathbb{M})}(\mathrm{cl}(C))$ has a finite basis, this basis is contained within the $\Gamma(\mathbb{M})$-span of finitely many elements of $\mathrm{cl}(C)$. But because $\mathrm{span}_{\Gamma(\mathbb{M})}(\mathrm{cl}(C)) = \mathrm{span}_{\Gamma(\mathbb{M})}(C)$ by the same reasoning as above, it follows that there is some finite $C_{1} \subset C$ whose $\Gamma(\mathbb{M})$-span contains the basis for $\mathrm{span}_{\Gamma(\mathbb{M})}(\mathrm{cl}(a)) \cap \mathrm{span}_{\Gamma(\mathbb{M})}(\mathrm{cl}(C))$, so $\mathrm{span}_{\Gamma(\mathbb{M})}(\mathrm{cl}(a)) \cap \mathrm{span}_{\Gamma(\mathbb{M})}(\mathrm{cl}(C)) \subset \mathrm{span}_{\Gamma(\mathbb{M})}(C_{1})$.

Now note that any element of the smallest closed set (in the sense of the proof of \ref{closure with respect to strong inclusion}) containing $C$ is, for some finite $C_{0} \subset C$, an element of the smallest closed set containing $C_{0}$. So because a set $S$ being closed is equivalent to $S \leq \mathbb{M}$ by the proof of Proposition \ref{closure with respect to strong inclusion}, $\mathrm{cl}(C)=\bigcup_{C_{0}\mathrm{\:finite}} \mathrm{cl}(C_{0})$. So because $\Gamma(\mathrm{cl}(aC_{1}))$ is finite (by the first claim above, because $aC_{1}$ is finite), there is some finite $C_{2} \subset C$ such that $\mathrm{cl}(C_{2})$ contains $\Gamma(\mathrm{cl}(aC_{1}) \cap \mathrm{cl}(C))$. We show that $C_{0} : = (C_{2}, C_{1})$ is as desired: $a \ind_{C_{0}} C$.

Let us first show, as a preliminary step, that for $C'_{0}:=(\Gamma(\mathrm{cl}(aC_{1}) \cap \mathrm{cl}(C)), C_{1})$, $a\ind_{C'_{0}}C$. We see that $\mathrm{cl}(aC'_{0}) \subset \mathrm{cl}(aC_{1})$: $aC'_{0} \subseteq \mathrm{cl}(aC_{1}) $, so because $\mathrm{cl}(aC_{1}) \leq \mathbb{M}$, $\mathrm{cl}(aC'_{0}) \subset \mathrm{cl}(aC_{1})$. So $\Gamma(\mathrm{cl}(aC'_{0})) \subseteq \Gamma(\mathrm{cl}(aC_{1})) $, and thus $\Gamma(\mathrm{cl}(aC'_{0}))\cap \Gamma(\mathrm{cl}(C)) \subseteq \Gamma(\mathrm{cl}(aC_{1})) \cap \Gamma(\mathrm{cl}(C)) = \Gamma(\mathrm{cl}(C'_{0})) $, so $\Gamma(\mathrm{cl}(aC'_{0}))\cap \Gamma(\mathrm{cl}(C)) = \Gamma(\mathrm{cl}(C'_{0})) $ (because $A \subseteq B$ implies $\mathrm{cl}(A) \subseteq \mathrm{cl}(B)$ by similar reasoning to before). Now it follows from the first claim above that $\mathrm{span}_{\Gamma(\mathbb{M})}(\mathrm{cl}(S))=\mathrm{span}_{\Gamma(\mathbb{M})}(V(S))$ for any set $S$, so to complete this preliminary step it remains to show that $\mathrm{span}_{\Gamma(\mathbb{M})}(aC_{1}) \cap \mathrm{span}_{\Gamma(\mathbb{M})}(C) = \mathrm{span}_{\Gamma(\mathbb{M})}(C_{1})$. Let $v = a_{1} + c_{1}$, for $a_{1} \in \mathrm{span}_{\Gamma(\mathbb{M})}(aC_{1})$, $c_{1} \in \mathrm{span}_{\Gamma(\mathbb{M})}(C_{1})$, be any element of $\mathrm{span}_{\Gamma(\mathbb{M})}(aC_{1})$. Suppose $v \in \mathrm{span}_{\Gamma(\mathbb{M})}(C)$ as well; then $a_{1} \in  \mathrm{span}_{\Gamma(\mathbb{M})}(CC_{1}) = \mathrm{span}_{\Gamma(\mathbb{M})}(C)=\mathrm{span}_{\Gamma(\mathbb{M})}(\mathrm{cl}(C))$, while $a_{1} \in \mathrm{span}_{\Gamma(\mathbb{M})}(a)=\mathrm{span}_{\Gamma(\mathbb{M})}(\mathrm{cl}(a))$. So $a_{1} \in \mathrm{span}_{\Gamma(\mathbb{M})}(\mathrm{cl}(a)) \cap \mathrm{span}_{\Gamma(\mathbb{M})}(\mathrm{cl}(C)) \subseteq \mathrm{span}_{\Gamma(\mathbb{M})}(C_{1})$, so $v = a_{1} + c_{1} \in \mathrm{span}_{\Gamma(\mathbb{M})}(C_{1})$. Thus $\mathrm{span}_{\Gamma(\mathbb{M})}(aC_{1}) \cap \mathrm{span}_{\Gamma(\mathbb{M})}(C) = \mathrm{span}_{\Gamma(\mathbb{M})}(C_{1})$, as we wanted. This gives us $a\ind_{C'_{0}}C$.

We now finally show that $a \ind_{C_{0}} C$. Here we will use symmetry, base monotonicity, and monotonicity, which we will prove later--of course, without assuming local character. First, $a \ind_{C'_{0}} C$ implies $\mathrm{cl}(aC'_{0}) \ind_{\mathrm{cl(C'_{0})}} \mathrm{cl}(C)$. Moreover, because $C'_{0} \subseteq \mathrm{cl}(C_{0})$, it is the case that $\mathrm{cl}(C_{0}') \subseteq \mathrm{cl}(C_{0})$, while similarly $\mathrm{cl}(C_{0}) \subseteq \mathrm{cl}(C)$. So from our assumption of base monotonicity to be proven later, $\mathrm{cl}(aC'_{0}) \ind_{\mathrm{cl}(C'_{0})} \mathrm{cl}(C)$ implies $\mathrm{cl}(aC'_{0}) \ind_{\mathrm{cl}(C_{0})} \mathrm{cl}(C)$, which implies $a \ind_{\mathrm{cl}(C_{0})} \mathrm{cl}(C)$ by symmetry and monotonicity, which implies $\mathrm{cl}(a \mathrm{cl}(C_{0})) \ind_{\mathrm{cl}(C_{0})} \mathrm{cl}(C)$, which since $\mathrm{cl}(a \mathrm{cl}(C_{0})) = \mathrm{cl}(aC_{0})$ (by similar reasoning to $\mathrm{cl}(C_{0}') \subseteq \mathrm{cl}(C_{0})$ above) implies $\mathrm{cl}(a C_{0}) \ind_{\mathrm{cl}(C_{0})} \mathrm{cl}(C)$. But this implies $a \ind_{C_{0}} C$, as needed to show local character.

\textit{Finite character:} As in the proof of local character, $\mathrm{cl} (BC) = \cup_{b \subset B \mathrm{\: finite}} \mathrm{cl} (bC).$ Now suppose that $A \nind _C B$. From this, finite character follows, but to be specific about the non-independence, one of two things must hold in order to witness this dependence:
\begin{itemize}
    \item There is some $g \in \Gamma (\mathrm{cl} (AC)) \cap \Gamma (\mathrm{cl} (BC)) $ which is not in $\mathrm{cl} (C)$. But then there is some $b \in B$ such that $g \in \mathrm{cl} (bC).$ But then already $g \in \Gamma ( \mathrm{cl} (AC)) \cap  \Gamma (\mathrm{cl} (bC))$ and so $A \nind _C b$. 

    \item There is some element $v \in \span _{Q_{\Gamma(\mathbb{M})}} (V(\mathrm{cl} (AC))) \cap \span _{Q_{\Gamma(\mathbb{M})}} (V(\mathrm{cl} (BC))) \setminus \span _{Q_{\Gamma(\mathbb{M})}} (V(\mathrm{cl} (C))) .$ But then one can write the element $v$ as a finite linear combination of elements in $V( \span _{Q_{\Gamma(\mathbb{M})}} (V(\mathrm{cl} (BC))) )$ with coefficients in $Q_{\Gamma(\mathbb{M})}$. But for each of the vector elements in this span, there are finitely many $b_i \in B$ such that the vector is in the closure of the union of the $b_i$. Take the union over the elements in the linear combination and call the resulting finite tuple $b$. Then it follows that $v \in \span _{Q_{\Gamma(\mathbb{M})}}(V(\mathrm{cl} (Ab)))$ and so $A \nind _C b$. 

\end{itemize}

\textit{Monotonicity:} The relation $\mathrm{cl} (-)$ is monotone by its definition, and so if there is some $B' \subseteq B$ and $\Gamma(\mathrm{cl} (AC) \cap \mathrm{cl} (B'C)) \setminus \Gamma(\mathrm{cl} (C))$ is nonempty then certainly $\Gamma(\mathrm{cl} (AC))) \cap \Gamma(\mathrm{cl} (BC)) \setminus \Gamma(\mathrm{cl} (C))$ is nonempty; similarly for $\mathrm{span}_{Q_{\Gamma(\mathbb{M})}}(V(\mathrm{cl} (AC))) \cap \mathrm{span}_{Q_{\Gamma(\mathbb{M})}}(V(\mathrm{cl} (B'C)))  \setminus \mathrm{span}_{Q_{\Gamma(\mathbb{M})}}(V(\mathrm{cl} (C))) $  and $\mathrm{span}_{Q_{\Gamma(\mathbb{M})}}(V(\mathrm{cl} (AC))) \cap \mathrm{span}_{Q_{\Gamma(\mathbb{M})}}(V(\mathrm{cl} (BC)))  \setminus \mathrm{span}_{Q_{\Gamma(\mathbb{M})}}(V(\mathrm{cl} (C))) $.   Thus $A \nind _C B'$ implies $A \nind _C B $.

\textit{Extension:} Let $A, B, C$ be sets, and we find $A' \equiv_{C} A$ with $A' \ind_{C} B$. By monotonicity and the ``only if" direction of Corollary \ref{quantifier elimination}, we may assume  $\mathrm{cl}(A) = A$, $\mathrm{cl}(B) = B$, and $\mathrm{cl}(C) = C$,  and $C \subset A, B$. Because $C \leq A, B$ and $B \leq \mathbb{M}$, we may find a basis $\langle c_{i} \rangle_{i \in \mathcal{I}}$ for $V(C)$ over $Q_{\Gamma(C)}$, which will also be independent over $Q_{\Gamma(B)}$, $Q_{\Gamma(A)}$ and $Q_{\Gamma(\mathbb{M})}$, and complete that to a basis $\langle c_{i} \rangle_{i \in \mathcal{I}}, \langle b_{i} \rangle_{i \in \mathcal{K}}$ for $B$ over $Q_{\Gamma(B)}$, which will be independent over $Q_{\Gamma(\mathbb{M})}$, and a basis $\langle c_{i} \rangle_{i \in \mathcal{I}}, \langle a_{i} \rangle_{i \in \mathcal{J}}$ for $A$ over $Q_{\Gamma(A)}$. Now construct an abstract structure $D$ extending $B$ as follows. Define $\Gamma(D) := G \sqcup  \Gamma(B)$, where $G$ satisfies the same quantifier-free type in the graph language over $\Gamma(C)$ that $\Gamma(A) \backslash \Gamma(C)$ satisfies over $\Gamma(C)$; the full graph structure on $G \Gamma(B)$ is chosen arbitrarily, subject to this condition. Then define $V(D)= \tilde{B} \oplus_{Q_{G \sqcup  \Gamma(B)}} \bar{A}$, where $\tilde{B}$ is a \textit{$Q_{G \sqcup  \Gamma(B)}$-module} with basis $\langle c_{i} \rangle_{i \in \mathcal{I}}, \langle b_{i} \rangle_{i \in \mathcal{K}}$ extending $V(B)$ ($\tilde{B}$ can be constructed because $\langle c_{i} \rangle_{i \in \mathcal{I}}, \langle b_{i} \rangle_{i \in \mathcal{K}}$ is already a $Q_{B}$-basis for $V(B)$), and $\bar{A}$ is a $Q_{G \sqcup  \Gamma(B)}$-module with basis consisting of new abstract elements $\langle a'_{i} \rangle_{i \in \mathcal{J}}$. Then $(Q_{G \sqcup \Gamma(C)},\mathrm{span}_{Q_{G \sqcup \Gamma(C)}}(\langle c_{i} \rangle_{i \in \mathcal{I}}, \langle a'_{i} \rangle_{i \in \mathcal{J}})) \leq D$, because its $V$-sort's basis $\langle c_{i} \rangle_{i \in \mathcal{I}}, \langle a'_{i} \rangle_{i \in \mathcal{J}}$ remains independent over  $Q_{G \sqcup  \Gamma(B)}= Q_{\Gamma(D)}$. Moreover, $B \leq D$ by construction of $\tilde{B}$. So by Lemma \ref{embedding property for strong inclusions}, there is $\iota: D \hookrightarrow \mathbb{M}$ with $\iota|_{B} = \mathrm{id}_{B}$ and $\iota(D) \leq \mathbb{M}$. Let $A':= \iota(G \sqcup \Gamma(C),\mathrm{span}_{Q_{G \sqcup \Gamma(C)}}(\langle c_{i} \rangle_{i \in \mathcal{I}}, \langle a'_{i} \rangle_{i \in \mathcal{J}}))) $; then $A' \leq \iota(D) \leq \mathbb{M}$, so $A' \leq \mathbb{M}$. There is an isomorphism between $A$ and $A'$ taking $\Gamma(A) = (\Gamma(A\backslash C)) \sqcup\Gamma(C)$ to $\iota(G \sqcup \Gamma(C))$ as enumerated sets, and taking the $Q_{\Gamma(A)}$-basis $\langle c_{i} \rangle_{i \in \mathcal{I}}, \langle a_{i} \rangle_{i \in \mathcal{J}}$ for $V(A)$ to the $Q_{\Gamma(A')}$-basis $\langle c_{i} \rangle_{i \in \mathcal{I}}, \langle \iota(a'_{i}) \rangle_{i \in \mathcal{J}}$ for $V(A')$; thus since $A, A' \leq \mathbb{M}$ and this isomorphism fixes $C$ pointwise, $A \equiv_{C} A'$ by Corollary \ref{quantifier elimination}. It remains to show $A' \ind_{C} B$--note here that $A= \mathrm{cl}(A)$, $B= \mathrm{cl}(B)$ and $C= \mathrm{cl}(C)$. By construction, $\Gamma(A') \cap \Gamma(B)= \Gamma(C)$. Moreover, $\langle c_{i} \rangle_{i \in \mathcal{I}}, \langle \iota(a'_{i}) \rangle_{i \in \mathcal{J}}, \langle b_{i} \rangle_{i \in \mathcal{K}}$ are independent over $Q_{\Gamma(\iota(D))}$ by construction, so because $D \leq \mathbb{M}$, they are independent over $Q_{\Gamma(\mathbb{M})}$. Thus $\mathrm{span}_{Q_{\Gamma(\mathbb{M})}}(V(A')) \cap \mathrm{span}_{Q_{\Gamma(\mathbb{M})}}(V(B)) = \mathrm{span}_{Q_{\Gamma(\mathbb{M})}}(V(C))$, because $\mathrm{span}_{Q_{\Gamma(\mathbb{M})}}(\langle c_{i} \rangle_{i \in \mathcal{I}}, \langle \iota(a'_{i}) \rangle_{i \in \mathcal{J}}) \cap \mathrm{span}_{Q_{\Gamma(\mathbb{M})}}(\langle c_{i} \rangle_{i \in \mathcal{I}}, \langle b_{i} \rangle_{i \in \mathcal{K}}) = \mathrm{span}_{Q_{\Gamma(\mathbb{M})}}(\langle c_{i} \rangle_{i \in \mathcal{I}})$. This completes the argument that $A' \ind_{C} B$.

\textit{Symmetry:} This follows directly from the definition.

\textit{Transitivity:} Let $A \ind_{D} C$ and $A \ind_{C} B$ with $D \subseteq C \subseteq B$; we show that $A \ind_{D} B$. As with extension, we may assume that $\mathrm{cl}(A) = A$, $\mathrm{cl}(B) = B$, $\mathrm{cl}(C) = C$, and $\mathrm{cl}(D) = D$, and that $D \subseteq A$.

To show $A \ind_{D} B$, let us first show that $\Gamma(A) \cap \Gamma(B) = \Gamma(D)$. $A \ind_{C} B$ implies that $\Gamma(\mathrm{cl}(AC)) \cap \Gamma(B)=\Gamma(C)$, so $\Gamma(A) \cap \Gamma(B) \subseteq \Gamma(C)$. But $\Gamma(A) \cap \Gamma(C)= \Gamma(D) $, so $\Gamma(A) \cap \Gamma(B)= \Gamma(D) $. Now let us show that $\mathrm{span}_{\Gamma(\mathbb{M})}(V(A)) \cap \mathrm{span}_{\Gamma(\mathbb{M})}(V(B)) = \mathrm{span}_{\Gamma(\mathbb{M})}(V(D))$. The relations $A \ind_{C} B$ and $A \ind_{D} C$ give $\mathrm{span}_{\Gamma(\mathbb{M})}(V(\mathrm{cl}(AC)) \cap \mathrm{span}_{\Gamma(\mathbb{M})}(V(B)) = \mathrm{span}_{\Gamma(\mathbb{M})}(V(C))$ and $\mathrm{span}_{\Gamma(\mathbb{M})}(V(A)) \cap \mathrm{span}_{\Gamma(\mathbb{M})}(V(C)) = \mathrm{span}_{\Gamma(\mathbb{M})}(V(D))$. But the first of these implies $\mathrm{span}_{\Gamma(\mathbb{M})}(V(A)) \cap \mathrm{span}_{\Gamma(\mathbb{M})}(V(B)) \subseteq \mathrm{span}_{\Gamma(\mathbb{M})}(V(C))$, which with the second gives $\mathrm{span}_{\Gamma(\mathbb{M})}(V(A)) \cap \mathrm{span}_{\Gamma(\mathbb{M})}(V(B)) = \mathrm{span}_{\Gamma(\mathbb{M})}(V(D))$.

\textit{Base monotonicity:}

Let $A \ind_{D} B$ with $D \subseteq C \subseteq B$; we show that $A \ind_{C} B$. As with extension, we may assume that $\mathrm{cl}(A) = A$, $\mathrm{cl}(B) = B$, $\mathrm{cl}(C) = C$, and $\mathrm{cl}(D) = D$, and that $D \subseteq A$.

We first observe that $A \ind_{D} C$, a consequence of $A \ind_{D} B$ (by monotonicity, which does not use this)  implies that $\Gamma(\mathrm{cl}(AC))=\Gamma(A) \cup \Gamma(C)$, and that $\mathrm{span}_{\Gamma(\mathbb{M})}(V(\mathrm{cl}(AC)))=\mathrm{span}_{\Gamma(\mathbb{M})}(V(A)V(C))$. First, we may let $\langle d_{i} \rangle_{i \in \mathcal{I}}$ be a basis for $V(D)$ over $Q_{\Gamma(D)}$; since $D \leq \mathbb{M}$ this remains independent over $Q_{\Gamma(\mathbb{M})}$, and particularly over $Q_{\Gamma(A)}$ and $Q_{\Gamma(C)}$. We may then complete it to a basis $\langle d_{i} \rangle_{i \in \mathcal{I}} \langle a_{i} \rangle_{i \in \mathcal{J}}$ for $V(A)$ over $Q_{\Gamma(A)}$, which as $A \leq \mathbb{M}$ remains independent over $Q_{\Gamma(\mathbb{M})}$, and a basis $\langle d_{i} \rangle_{i \in \mathcal{I}} \langle c_{i} \rangle_{i \in \mathcal{K}}$ for $V(C)$ over $Q_{\Gamma(C)}$, which as $C \leq \mathbb{M}$ remains independent over $Q_{\Gamma(\mathbb{M})}$. Because $A \ind_{D} C$ implies $\mathrm{span}_{\Gamma(\mathbb{M})}(V(A)) \cap \mathrm{span}_{\Gamma(\mathbb{M})}(V(C)) = \mathrm{span}_{\Gamma(\mathbb{M}}(V(D))$, $\mathrm{span}_{\Gamma(\mathbb{M})}(\langle d_{i} \rangle_{i \in \mathcal{I}}, \langle a_{i} \rangle_{i \in \mathcal{J}}) \cap \mathrm{span}_{\Gamma(\mathbb{M})}(\langle d_{i} \rangle_{i \in \mathcal{I}}, \langle c_{i} \rangle_{i \in \mathcal{K}}) = \mathrm{span}_{\Gamma(\mathbb{M})}(\langle d_{i} \rangle_{i \in \mathcal{I}})$, so $\langle d_{i} \rangle_{i \in \mathcal{I}}, \langle a_{i} \rangle_{i \in \mathcal{J}}, \langle c_{i} \rangle_{i \in \mathcal{K}} $ must be independent over $Q_{\Gamma(\mathbb{M})}$. Then $AC \subset (\Gamma(A) \sqcup \Gamma(C), \mathrm{span}_{Q_{\Gamma(A)\cup \Gamma(C)}}(\langle d_{i} \rangle_{i \in \mathcal{I}}, \langle a_{i} \rangle_{i \in \mathcal{J}}, \langle c_{i} \rangle_{i \in \mathcal{K}}) \leq \mathbb{M}$ because $\langle d_{i} \rangle_{i \in \mathcal{I}}, \langle a_{i} \rangle_{i \in \mathcal{J}}, \langle c_{i} \rangle_{i \in \mathcal{K}}$ is a basis for the $V$-sort over $Q_{\Gamma(A)\cup \Gamma(C)}$ and independent over $Q_{\Gamma(\mathbb{M})}$, and

$$\mathrm{cl}(AC) \supseteq (\Gamma(A) \cup \Gamma(C), \mathrm{span}_{Q_{\Gamma(A)\sqcup \Gamma(C)}}(\langle d_{i} \rangle_{i \in \mathcal{I}}, \langle a_{i} \rangle_{i \in \mathcal{J}}, \langle c_{i} \rangle_{i \in \mathcal{K}})),$$

so 

$$\mathrm{cl}(AC) = (\Gamma(A) \cup \Gamma(C), \mathrm{span}_{Q_{\Gamma(A)\cup \Gamma(C)}}(\langle d_{i} \rangle_{i \in \mathcal{I}}, \langle a_{i} \rangle_{i \in \mathcal{J}}, \langle c_{i} \rangle_{i \in \mathcal{K}}))$$

and particularly $\Gamma(\mathrm{cl}(AC))=\Gamma(A) \cup \Gamma(C)$ and 

$$\mathrm{span}_{\Gamma(\mathbb{M})}(V(\mathrm{cl}(AC)))=\mathrm{span}_{\Gamma(\mathbb{M})}(\langle d_{i} \rangle_{i \in \mathcal{I}}, \langle a_{i} \rangle_{i \in \mathcal{J}}, \langle c_{i} \rangle_{i \in \mathcal{K}})=\mathrm{span}_{\Gamma(\mathbb{M})}(V(A)V(C)).$$

Now we show, from $A \ind_{D} B$, that $A \ind_{C} B$. First, we must show that $\Gamma(\mathrm{cl}(AC)) \cap \Gamma(B)=\Gamma(C)$. But by $\Gamma(\mathrm{cl}(AC))=\Gamma(A) \cup \Gamma(C)$, this is the same thing as to show that $(\Gamma(A) \cup \Gamma(C)) \cap \Gamma(B)=\Gamma(C)$, which follows from $\Gamma(A) \cap \Gamma(B) = \Gamma(D)$. Second, we must show that $\mathrm{span}_{\Gamma(\mathbb{M})}(V(\mathrm{cl}(AC)) \cap \mathrm{span}_{\Gamma(\mathbb{M})}(V(B)) = \mathrm{span}_{\Gamma(\mathbb{M})}(V(C))$. But by $\mathrm{span}_{\Gamma(\mathbb{M})}(V(\mathrm{cl}(AC)))=\mathrm{span}_{\Gamma(\mathbb{M})}(V(A)V(C))$, this is the same thing as $\mathrm{span}_{\Gamma(\mathbb{M})}(V(A)V(C)) \cap \mathrm{span}_{\Gamma(\mathbb{M})}(V(B)) = \mathrm{span}_{\Gamma(\mathbb{M})}(V(C))$. We know that $\mathrm{span}_{\Gamma(\mathbb{M})}(V(A)) \cap \mathrm{span}_{\Gamma(\mathbb{M})}(V(B)) = \mathrm{span}_{\Gamma(\mathbb{M})}(V(D))$; let $v = a+ c \in \mathrm{span}_{\Gamma(\mathbb{M})}(V(A)V(C)) $ for $a \in \mathrm{span}_{\Gamma(\mathbb{M})}(V(A)), c\in  \mathrm{span}_{\Gamma(\mathbb{M})}(V(C))$. Suppose that $v \in \mathrm{span}_{\Gamma(\mathbb{M})}(V(B))$ as well; we show $v \in \mathrm{span}_{\Gamma(\mathbb{M})}(V(C))$. Since $a+ c \in \mathrm{span}_{\Gamma(\mathbb{M})}(V(B)) $, subtracting, $a \in \mathrm{span}_{\Gamma(\mathbb{M})}(V(B))$. So $a \in \mathrm{span}_{\Gamma(\mathbb{M})}(V(B)) \cap \mathrm{span}_{\Gamma(\mathbb{M})}(V(A)) = \mathrm{span}_{\Gamma(\mathbb{M})}(V(D)) \subseteq \mathrm{span}_{\Gamma(\mathbb{M})}(V(C))$. So $v \in a+ c \in \mathrm{span}_{\Gamma(\mathbb{M})}(V(C))$, as desired.

\textit{The independence theorem:} 

We must prove the following proposition:

\begin{prop}
    If $M \prec \mathbb{M}$, $A_{1} \ind_{M} B$, $A_{2} \ind_{M} C$, $C \ind_{M} B$, and $A_{1} \equiv_{M} A_{2}$, then there is some $A \ind_{M} BC$ with $A \equiv_{MB} A_{1}$ and $A \equiv_{MC} A_{2}$.
\end{prop}

\begin{proof}
    Suppose that we have closed sets $D, E, F \leq \m M$ with $F \subseteq D \cap E$ and $D \ind _F E$. Take a $Q_{ \Gamma (F)}$-basis of $V(F)$, call it $B$. Then take an extension of this basis, $B \cup B_D$ to be a basis of $V(D)$ over the ring $Q _{\Gamma (D)}$. Similarly, let $B \cup B_E$ be a basis of $V( E)$ over $Q _{\Gamma (E)}$.

        Then we claim that $B \cup B_E \cup B_D$ is an independent set over $Q_{\Gamma (\m M)}$, and that the closure $\mathrm{cl}(DE) = (\Gamma (D) \cup \Gamma (E), \span _{Q_{\Gamma (D) \cup \Gamma (E)}} (B \cup B_E  \cup B_D))$. But this is just as in the proof of base monotonicity.



By our assumptions and definition of $\ind$, we have 
$$\mathrm{cl}(MA_1) \ind_M \mathrm{cl} (MB_0) , \, \mathrm{cl}(MA_2) \ind _M \mathrm{cl}(MC), \, \mathrm{cl}(MC) \ind _M \mathrm{cl} (MB).$$ (Here, note that $\mathrm{cl}(M) = M$, because $M \prec \mathbb{M}$ and, by definition of $\leq$, this implies $M \leq \mathbb{M}$.) By our elementary equivalence assumption, we have an isomorphism $f: \mathrm{cl}(MA_1) \rightarrow \mathrm{cl}(MA_2)$ which fixes $M$ pointwise and sends $A_1 $ to $A_2$. Pick a $Q _ { \Gamma (M)}$ basis $B_M$ of $V(M)$. Extend this basis to bases: 
\begin{enumerate}
    \item $B_M \cup B_{A_1} $ of $V(\mathrm{cl}(MA_1))$ over $Q_{\Gamma (\mathrm{cl}(MA_1))}$
    \item $B_M \cup f(B_{A_1}) $ of $V(\mathrm{cl}(MA_2))$ over $Q_{\Gamma (\mathrm{cl}(MA_2))}$
    \item $B_M \cup B_0 $ of $V(\mathrm{cl}(MB))$ over $Q_{\Gamma (\mathrm{cl}(MB))}$
    \item $B_M \cup C_0 $ of $V(\mathrm{cl}(MC))$ over $Q_{\Gamma (\mathrm{cl}(MC))}$
\end{enumerate}
By the hypothesis; and our claim following from the arguments from the proof of base monotonicity, we have the following: $B_M \cup B_{A_1} \cup B_0 $, $B_M \cup f(B_{A_1}) \cup C_0 $ and $B_M \cup C_0 \cup B_0 $ are each $Q_{\Gamma (\m M)}$-independent sets. 

Pick an embedding of $\Gamma (\mathrm{cl}( MA_1))$ into $\Gamma(\m M)$ which fixes $\Gamma (M)$ pointwise. Call it $\sigma$. And make sure that we have the property that $\sigma (\Gamma ( \mathrm{cl}(MA_1)))$ is disjoint from $\Gamma (\mathrm{cl}(MB))$ and $\Gamma(\mathrm{cl}(MC))$ except for $\Gamma (M)$. This follows via saturation, but we want a bit more. Now for $b \in \Gamma ( \mathrm{cl}(MB))$, $c \in \Gamma (\mathrm{cl}(MC))$ and $a \in \Gamma (\mathrm{cl}(MA_1))$ we also demand that
$$E(\sigma ( a ), b ) \leftrightarrow E(a, b)$$ and also that $$E(\sigma (a), c ) \leftrightarrow E( f(a) , c ). $$ This all agrees on the portion which overlaps, since $\mathrm{cl}(MC) \ind _M \mathrm{cl}(MB)$ means that the graph portions of $\mathrm{cl}(MC)$ and $\mathrm{cl}(MB)$ are disjoint except for the graph portion of $M$. Then by saturation, we have such a $\sigma.$  

This $\sigma$ induces an embedding of $Q_{\Gamma ( \mathrm{cl}(MA_1))}$ into $Q_{\sigma(\Gamma (\mathrm{cl}((MA_1))))}$. Now we want to extend $\sigma $ to the module sort. All of the bases which have appeared so far are small relative to the saturation of $\m M$, and so we can pick some $B_{A_1}^*$ with the property that $B_M \cup B_0 \cup C_0 \cup B_{A_1}^*$ is independent over $Q_{\Gamma (\m M)}$. 

Now $(\sigma (\Gamma (\mathrm{cl}(MA_1)))),\span _ {Q_{\sigma (\Gamma (\mathrm{cl}(MA_1)))}} (B_M \cup B_{A_1}^*) $
is a strong closed substructure of $\m M$. 
Define $\sigma^*$
a map from $\mathrm{cl}(MA_1)$
to $\left( \sigma (\Gamma (\mathrm{cl}(MA_1))),\span_{Q_{\sigma (\Gamma (MA_1)))}} (B_M \cup B_{A_1}^*) \right)$ via $\sigma $ on the graph sort. On the  module sort, note that $\sigma $ induces an isomorphism of the associated division rings, and we extend by linearity the above given map of the bases. Call the image of $A_1$ under this map $A_1^*.$ 

The closed structure $\left( \sigma (\Gamma (\mathrm{cl}(MA_1)))),\span _ {Q_{\sigma (\Gamma (\mathrm{cl}(MA_1)))}} (B_M \cup B_{A_1}^* ) \right)$ is the closure of $MA_{1}^*.$ By the claim at the beginning of this proof, $\mathrm{cl}(MA_1B) $ has graph sort $\Gamma ( \mathrm{cl}(MA_1)) \cup \Gamma (\mathrm{cl}(MB))$ and the module sort has basis $B_M \cup B_{A_1} \cup B_0.$ 
Similarly, $\mathrm{cl}(MA_1^*B)$ has graph sort $\sigma (\Gamma (\mathrm{cl}(MA_1))) \cup \Gamma (\mathrm{cl}(MB))$ and basis $B_M \cup B_{A_1}^* \cup B_0$. 
Putting these together, $\sigma $ extended by the identity on $\Gamma (\mathrm{cl}(MB))$ is a graph isomorphism between the graphs of these two substructures, while the above defined map of bases and its extension by linearity gives an isomorphism of the module sorts. Therefore, 
$$\mathrm{cl}(MA_1B) \cong _{MB} \mathrm{cl}(MA_1^*B)$$ and now we know by our characterization of types that $A_1 \equiv _{MB} A_1^*$. The analogous argument gives that $A_1^* \equiv _{MC} A_2$. 

The independence of $A_1^*$ from $BC$ over $M$ follows from the fact that $B_M \cup B_0 \cup C_0 \cup B_{A_1}^*$ is independent and the graphs $\Gamma(\mathrm{cl}(MA_1^*)) $ and $\Gamma ( \mathrm{cl}(MBC))$ have intersection $\Gamma (M)$. 
\end{proof}

Having proven that $\ind$ satisfies the hypotheses of Fact \ref{Kim-Pillay theorem}, we conclude that forking-independence coincides with $\ind$ in $T^{\neg \mathrm{StFk}}$, and $T^{\neg \mathrm{StFk}}$ is simple.

\subsection{Instability of the independence relation in $T^{\neg\mathrm{StFk}}$}\label{instability of the independence relation}

Let $v_0, \ldots v_3 \in V(\mathbb{M})$ be linearly independent over $Q_{\Gamma(\mathbb M)}$. For each $i \in I( \mathbb M) $ and $j \in J ( \mathbb M)$ we define the following: 
\begin{eqnarray}
u_i:=v_0+X_i v_1,\\
t_i:=v_2+X_i v_3,\\
w_j:=v_0+X_j v_2,\\
z_j:=v_1+X_j v_3,
\end{eqnarray}

\begin{lemma}
    Let $i \in I ( \m M) $ and $j \in J( \m M)$. Let $U_i:= \span_{Q_{\Gamma( \m M) }} ( u_i , t_i )$ and $W_j = \span_{Q_{\Gamma( \m M) }} ( w_j , z_j )$. When $\neg E(i,j)$ holds, $U_i \cap W_j = \{ 0 \}$. When $E(i,j)$ holds, we have that $U_i \cap W_j = \span_{Q_{\Gamma( \m M) }} ( u_i + X_jt_i ).$ 

    It follows that if one takes $\alpha_i = ( u_i , t_i )$ and $\beta _j = ( w_j , z_j )$ then $\alpha_ i \ind \beta _j $ if and only if $\neg E(i,j)$. 
\end{lemma}

\begin{proof}
    Suppose that for some $i, j$ and some $r,s,p,q \in Q_{\Gamma( \m M) }$ we have that 
    $$r u_i + s t_i  = p w_j + q z_j .$$
Then note that by the above definitions, we have that 
$$r (v_0+X_i v_1) + s (v_2+X_i v_3)  = p (v_0+X_j v_2) + q (v_1+X_j v_3).$$
Now, noting that the $v_i$ are linearly independent, we must have that:
\begin{eqnarray}
    r=p \\
    rX_i=q \\
    s = p X_j \\
    sX_i = q X_j.
\end{eqnarray}
But then by substitutions, one can see that $rX_j X_i = r X_i X_j $ so $r (X_j X_i -X_i X_j) = 0$. Now, if $\neg E(i,j)$, then we have that $X_j X_i -X_i X_j = 2 X_j X_i \neq 0$, so it must be that $r = 0$. But then by the above equations, it follows that each of $s,p,q$ must also be zero. It thus follows that $U_i \cap W_j = \{0 \}$. 

On the other hand, if $E(i,j)$ holds, then $X_i X_j = X_j X_i$ and so we have that 
\begin{eqnarray*}
    u_i + X_j t_i = v_0 + X_i v_1 +X_j v_2 + X_j X_i v_3 \\
    = v_0 +X_j v_2 + X_i v_1 +X_i X_j v_3 \\ 
    = w_j + X_i z_j.
\end{eqnarray*}

Note that both sides are nonzero by linear independence of $v_0, v_2, v_3, v_3$. Now $\Gamma (\mathrm{cl} (\alpha_i)) = \emptyset $ and $V (\mathrm{cl} (\alpha_i)) = \span_ {\mathbb{Q}} (u_i, t_i)$. (We know $(\emptyset, \span_ {\mathbb{Q}} (u_i, t_i)) \leq \mathbb{M}$ because $u_i, t_i$ form a basis for $\span_ {\mathbb{Q}} (u_i, t_i))$ over $Q_{\emptyset}=\mathbb{Q}$ while remaining independent over $Q_{\Gamma(\mathbb{M})}$.) Similarly, 
$\Gamma (\mathrm{cl} (\beta_j)) = \emptyset $ and $V (\mathrm{cl} (\alpha_j)) = \span_ {\mathbb{Q}} (w_j, z_j)$. But then it follows that $\alpha_i \ind \beta_{j}$ if and only if $U_i \cap W_j = \{ 0 \}$ and from the above calculation this occurs precisely when $\neg E(i,j)$ holds. 

\end{proof}

Now by $\kappa = \lambda^{+}$-saturatedness of $\mathbb{M}$, we may find $\{i_{\gamma}, j_{\gamma}\}_{\gamma < \lambda}$ with $i_{\gamma} \in I(\mathbb{M})$, $j_{\gamma} \in J(\mathbb{M})$ and $E(i_{\gamma_{1}}, j_{\gamma_{2}})$ if and only if $\gamma_{1} < \gamma_{2}$. Then with $\alpha_{i}, \beta_{i}$ as in the lemma, $\alpha_{i_{\gamma_{1}}} \ind \beta_{j_{\gamma_{2}}}$ if and only if $\gamma_{2} < \gamma_{1}$. Recall that $\lambda$ was large enough to apply the Erdős–Rado theorem, so applying this, we obtain an indiscernible sequence $\{a_{i}, b_{i}\}_{i < \omega}$ such that $a_{i} \nind b_{j}$ if and only if $i < j$. In other words, the forking relation is unstable (in the sense of \cite{PW13}, Section 4).

As explicated in Remark 7.5 of \cite{KoponenConjecture}, this implies that $T^{\neg \mathrm{StFk}}$ is a counterexample to the stable forking conjecture. We recall the proof: let $ab \models \mathrm{tp}(a_{i}b_{j})$ for $i < j$, so $a \nind b$. We show that there is no stable formula $\varphi(x, y)$ with $a \models \varphi(x, b)$ and such that $\varphi(x, b)$ forks over the $\emptyset$. Otherwise, because $a \models \varphi(x, b)$ for  $ab \models \mathrm{tp}(a_{i}b_{j})$ with $i < j$ and $\{a_{i}, b_{i}\}_{i < \omega}$ is indiscernible, $\models \varphi(a_{i}, b_{j})$ for $i < j$. On the other hand, because $a_{i} \ind b_{j}$ for $i \geq j$ and $\varphi(x, b)$, and hence $\varphi(x, b_j))$, forks over $\emptyset$, $\models \neg \varphi(a_{i}, b_{j})$ for $i \geq j$. So $\varphi(x, y)$ is unstable, a contradiction. We conclude that $T^{\neg\mathrm{StFk}}$ is a counterexample to the stable forking conjecture.

\begin{remark}
    In the proof of local character in verifying the hypotheses of Fact \ref{Kim-Pillay theorem}, we showed more: for any finite set $a$ and set $C$, there is some \textit{finite} $C_{0} \subset C$ such that $a \ind_{C_{0}} C$ (local character only demands that $C_{0}$, in the case of a theory with countable language, be countable.) Since $\ind$ coincides with forking-independence, $T^{\neg\mathrm{StFk}}$ is even a \textit{supersimple} counterexample to the stable forking conjecture. However, some main cases of the stable forking conjecture remain open:

    \begin{question}
    
        \begin{itemize}
            \item Do supersimple theories of finite $\mathrm{SU}$-rank satisfy the conclusion of the stable forking conjecture?

            \begin{itemize}
                \item Is the conclusion of the stable forking conjecture satisfied for $a, b  \subset \mathbb{M} \models T$, for a supersimple theory $T$ and $a\nind b$, where $\mathrm{SU}(a)=\mathrm{SU}(b) = 3$?\footnote{As stated in the introduction, Brower proves the case where $\mathrm{SU}(a) = 2$ and $\mathrm{SU}(b) < \omega$, improving on work of Peretz, so this is the least-rank case that remains open.}
            \end{itemize}
            \item Do countably categorical simple theories satisfy the conclusion of the stable forking conjecture?

            \item Is there a counterexample to the stable forking conjecture that does not interpret a group (or such that there is no type-definable group in $T^{eq}$?)
        \end{itemize}
    \end{question}

\end{remark}

\begin{tcolorbox}[
    enhanced,
    breakable,
    oversize,                 
    colback=blue!5!white,     
    colframe=blue!75!black,   
    skin first=enhanced,      
    skin middle=enhanced,     
    skin last=enhanced        
]

\textsf{\textbf{Sidebar: Ziegler's question}}

\:

\: The main result of \cite{casanovaspotier2018} is that $T$ has stable forking if and only if $T^{eq}$ has stable forking. One might view the result as part of a general program to understand when a theory has stable forking. Several other results, some cited below, give other reductions. For instance, \cite{kim2001around} shows that the conclusion of the stable forking conjecture holds of a theory if and only if it holds for types over models. Asking about another possibility of this general motif, Casanovas and Potier \cite{casanovaspotier2018} conclude their paper with: \emph{Assume all 1-types in T have stable forking. Does it follow that T has stable forking? A positive answer would be very helpful.} They attribute the question to Ziegler. We show that the answer to this question is negative. 

\subsection{A simple example of forking for singletons in $T^{\neg \mathrm{StFk}}$}
Consider the formula $$w = \exists x \in \Gamma , \, w= \pm x v.$$
\: Notationally, in what follows, we let equalities in which $\pm$ appears formally denote the disjunction of the resulting equations in which any combination of the signs are chosen. 

\: By our characterization of forking, $w \nind v$ (this is a disjunction of the form $w$ is in the global span of $v$). 
Let's explain why this disjunction is a stable formula. 

\: Suppose that we have an indiscernible sequence $\{b_i , c_i\}_{i < \omega}$ such that $\phi(b_i , c_j)$ if and only if $i <j$. Then let $x_{ij}$ be such that $\pm x_{ij} c_j = b_i$. 

\: Then $\pm x_{12} c_2 = \pm x_{13} c_3 = b_1$ and $\pm x_{02} c_2 = \pm x_{03} c_3 =b_0.$ But then we can see that $\pm x_{12}/x_{13} = \pm x_{02}/x_{03}$. But then we must have that $\pm x_{12} = x_{02}$ and $\pm x_{13} = x_{03}.$ But then, for instance, $b_0 = \pm b_1$, a contradiction. This is the simplest non-misleading example of a fact that we will establish in this section: the conclusion of the stable forking conjecture is true for $1$-types in $T^{\neg \mathrm{StFk}}$. Combined with our main results above, this yields a negative answer to the question of Ziegler.

\subsection{Forking for singletons in $T^{\neg \mathrm{StFk}}$}

The following result is well-known and established in several existing places in the literature in more general forms as discussed below. We include it to keep the exposition here more self-contained. 

\begin{lemma}
A simple theory $T$ has stable forking if and only if the conclusion of stable forking holds for types over $A$ and $B$ with $A \subset B$ algebraically closed. Moreover the conclusion of the stable forking holds for tuples of a given length, $\bar x$ if and only if the conclusion of the stable forking conjecture holds for types in $\bar x$ over algebraically closed $A \subset B.$ 
\end{lemma}
\begin{proof}
    Suppose that $A \subset B$ and $a \nind _A B$. Now $a \ind _A \mathrm{acl}(A)$ and by base monotonicity and transitivity it follows that $a \ind _A B$ holds if and only if $a \ind _{\mathrm{acl}(A)} \mathrm{acl} (B)$. Now suppose we have the conclusion of the stable forking conjecture over algebraically closed sets for types in $x$ (we are proving the only nontrivial direction of the above equivalence in the more general case). So, we have an element $b$ from $\mathrm{acl}(B)$ with the property that $\phi(a,b)$ holds,  $\phi(x,b)$ forks over $\mathrm{acl}(A)$, and $\phi(x, y)$ is stable. 

    Because $b$ is algebraic over $B$, there is some $b_0 \in B$ and a formula $\theta(y,b_0)$ such that the elements which satisfy $\theta$ are the conjugates of $b$ under automorphisms fixing $B$. By Lemma 1.2 of \cite{casanovaspotier2018} whenever $\phi (x,y)$ is stable and $\theta (v,x) $ implies $\exists ^{=n} x \theta (v,x)$ then $\psi (x,z) = \exists y \theta (y,z) \wedge \phi(x,y)$ is stable. 

    Now $\psi (x, b_0)$ is satisfied by $a$, is an instance of the stable formula $\psi (x,z)$, and implies a finite disjunction of formulas forking over $\mathrm{acl}(A)$, so forking over $A$, as desired.
\end{proof}

\: The above lemma, in closely related form - basically as well as its proof - appears in \cite{casanovaspotier2018}. This is a special case of a result of \cite{kim2001around} which shows that the conclusion of the stable forking conjecture holds of a theory if and only if it holds for types over models (see also Remark 1.1 of \cite{casanovaspotier2018}). 

\: We are then reduced to considering $B \subset C$ algebraically closed subsets and $a \nind _B C$. If $a \in \Gamma(\mathbb{M})$, then $\mathrm{acl}(aB) = a \mathrm{acl}(B) =a B$, so $a \in C$ but $a \notin B$, and since equality is a stable forking formula, we need only consider the case in which $ a \in V$. There are two possibilities for the forking in the case that $a \in V$. 
\begin{enumerate}
    \item First, it could be that $a = \sum f_i b_i$ where the $f_i \in Q_{\Gamma(C)}$ and $b_i \in V(B)$ are linearly independent and there is a graph vertex $g_c \in \Gamma(C) \backslash \Gamma(B)$ with $g_c$ contained in the support of the tuple of $f_i$'s.
    \item If this case doesn't hold, then we have $a \in \mathrm{span}_{\Gamma(\mathbb{M})} (V(C)) \setminus \mathrm{span}_{\Gamma(\mathbb{M})} (V(B))$, exhibited by $a = \sum f_i b_i$.
\end{enumerate} 



\: In case 1) $g_{c} \in (\mathrm{acl}(C) \cap \mathrm{acl}(Ba)) \backslash \mathrm{acl}(B)$, so by the above lemma and stability of $=$ (see also \cite{Kim01}, Theorem 4.3) the conclusion is satisfied.

\: In case 2), let us take above the tuple of the vector sort to be $\bar w$ and the tuple from the ring to be $\bar d$. Then $a = \sum _{k=1} ^n d_k w_k $. Let $S$ be the support of the tuple $\bar d$. Consider the square field extension of $\m Q$ generated by the squares of $X_{s}$ for $s \in S$, and denote this $\m Q_{S^2}$. This field is in the center of the ring $\m Q_{S}$ and $\m Q_{S}$ is a finite dimensional module over the squares. 

\: By clearing denominators, one could write the elements of $\m Q_S$ as a sum over monomials in $X_s$ for $s \in S$ with exponents at most one and coefficients in $\m Q_{S^2}$ - in fact, we can assume there is a single denominator in $\m Q_{S^2}$ and we take the sum to involve polynomials in the squares of the support. 

\: So, now write 
$$d_k = \frac{1}{g_k} \sum f_{k,e} X^e$$
where $e$ is a binary vector of length $|S|$ and $X^e:=X_{s_1}^ {e_1} \ldots X_{s_n} ^{e_n}$, for $e_{i} \in \{0, 1\}$ and $g_{k}, f_{k, e} \in Q_{S^{2}}$

\: Consider all $\eta$ to be a binary vector of length equal to $2^{|S|}$. Then consider $$d_k^\eta:=  \frac{1}{g_k} \sum (-1)^{\eta_e} f_{k,e} X^e.$$ 


\: Take a tuple of new variables $\bar z$ of length $|S|$. 
\: Let $X_{\bar z}$ denote the monomials with $\bar z$ in place of the elements of $S$ in the natural way. 
\: We consider the formula $\Phi (x, \bar y)$ which is the formula that specifies $\bar y $ is not the zero vector and then takes the disjunction over all $\eta$ as above of the formulas $\exists z_1, \ldots , z_{|S|} \left( g_k (X_{\bar z} ) x =  \sum (-1)^{\eta_e} f_{k,e} (X_{\bar z}) X_{\bar z}^e y_k  \right)$.


\: Now, this formula (with $a$ in place of $x$ and $\bar b $ in place of $\bar y$) holds of $a$ over $C$. The formula has a number of other important properties. First, the changes of sign induced by different $\eta$ do not affect which terms $d^{\eta}_k$ are zero (in fact, we can assume without loss of generality that none of these terms are zero and that this is part of the data specified by the formula). These changes also preserve the support of the coefficients in the linear sum. We'll argue that the formula $\Phi (x, \bar y)$ is stable.



\: So, take an indiscernible sequence $(a_i, \bar b_i )$ over the empty set with the property that $\Phi(a_i, \bar b_j )$ holds if and only if $i <j$. Now, any initial segment of the $a_i$ is in the span of any later $\bar b _j$. It follows that all of the $a_i$ in the entire sequence lie in a vector space of dimension at most $n$ where $n$ is the length of the tuple $\bar b_j$. So, let us pick a basis of the vector space spanned by the $a_{i}$ given by $e_1, \ldots e_r$ where $e_k = a_{i_k}$. Assume the basis is chosen in increasing order and let $N=i_{r}$. Now we are going to consider only the portion of the sequence after the initial segment of size $N$. 

\: The sequence that remains is indiscernible over our basis $e_1, \ldots , e_r$. 

\: Next, let $a_i = \sum^{r}_{t=1} \alpha _{ti} e_t$. This is possible, since the $e_i$ formed a basis. Similarly, let $e_t = \sum_{k=1} ^n \beta _{ktj} b_{kj}$ which is possible, since the basis must be in the span of any of the later $\bar b_j$ after the initial segment of length $N$. Taking $$d_{kij} = \sum_{t=1}^ r \alpha _{ti} \beta _{ktj}$$ we have that $a_i = \sum _{k=1}^ n d_{kij} b_{kj} $. Some caution here is that the expressions for $d_{kij}, \alpha_{ti}$ might have larger degree than our original expressions appearing in the formula $\Phi$ as well as the original linear combination from which $\Phi $ was defined. 

\: Now for each $i$ we consider the set $S_i$ given by the support of $\alpha_{ti}$ for all $t$ and $\beta_{kti}$. To be clear, this is simply a (finite) collection of vertices in our graph. By indiscernibility, the intersection of any two of the sets is some fixed set $G_0$ and $S_i = G_0 \cup T_i$ and any two of the $T_i$ are disjoint. Now, order the variables associated with the graph parameters in some way so that $G_0$ comes before $T_j$ which comes before $T_i$. 

\: Consider now our expression $$d_{kij} = \sum_{t=1}^ r \alpha _{ti} \beta _{ktj}.$$ The $\alpha _{ti}$ use expressions in $G_0 \cup T_i$ while the $\beta _{ktj}$ use expressions in $G_0 \cup T_j$. The only relations that possibly vary with the choice of $i,j$ are those between $T_i$ and $T_j$.

\: The rational expressions for $\alpha_{ti}$ and $\beta_{kti}$ are the same for every $i$: passing from $i$ to $j$ simply replaces each generator indexed by a vertex of $S_i$ with the generator indexed by the corresponding vertex of $S_j$. Enumerate $S_i$ at one index and choose rational expressions for these coefficients. An automorphism fixing the basis and carrying $(a_i,\bar b_i)$ to $(a_j,\bar b_j)$ gives an enumeration of $S_j$. Applying the expressions with the enumerated generators of $S_j$ in place of $S_i$ gives the corresponding expression at the $j$ index. By uniqueness of these coefficients (via linear independence of the vector sort elements we are considering), the resulting expressions necessarily give $\alpha_{tj}$ and $\beta_{ktj}$. By indiscernibility, every vertex of $G_0$ can wlog be assumed to occupy the same position in the enumeration at each index. The substitutions therefore fix the generators indexed by $G_0$. They also precisely preserve the relations on the $S_i$. 

\: Fix $i<j$. Replace each generator indexed by $T_i$ with the generator at the corresponding position in $T_j$, and each generator indexed by $T_j$ with the generator at the corresponding position in $T_i$. Leave the generators indexed by $G_0$ fixed. These substitutions turn the expressions for $\alpha_{ti}$ and $\beta_{ktj}$ into those for $\alpha_{tj}$ and $\beta_{kti}$. They can, however, change whether a generator from $T_i$ commutes (or anticommutes) with a generator from $T_j$. We must check exactly how this affects
$$
d_{kij}=\sum_{t=1}^r\alpha_{ti}\beta_{ktj}
\qquad\hbox{and}\qquad
d_{kji}=\sum_{t=1}^r\alpha_{tj}\beta_{kti}.
$$

\: Write each factor as a sum of ordered monomials with exponents zero or one and coefficients rational in the squares. Order the generators as $G_0<T_j<T_i$ in the first expression and in the corresponding order $G_0<T_i<T_j$ in the second. When ordering a product of two monomials, all commutations involving $G_0$ give the same signs before and after substitution. The only signs that can differ arise when a generator from $T_i$ in the first factor swaps order with a generator from $T_j$ in the second factor in the enumerations.

\: Fix the exponents, modulo two, of the resulting monomial. The exponents on $T_i$ come entirely from the first factor, and those on $T_j$ entirely from the second. These exponents therefore determine which of the commutations just described occur. Every term contributing to this monomial acquires the same additional sign, independently of $t$ and $k$. How the exponents on $G_0$ are divided between the two factors makes no difference to that sign. The coefficient of the whole monomial consequently changes by this one sign, with each square in its rational coefficient replaced by the square of the corresponding generator. In particular, terms cancel before substitution if and only if they cancel afterwards. 

\: Now choose graph vertices $\bar z$ and a sign choice $\eta$ for which the disjunct of $\Phi(a_i,\bar b_j)$ holds. Since $\bar b_j$ is independent, the coefficient of each $b_{kj}$ is unique, so
$$
d_{kij}:=\sum_{t=1}^r\alpha_{ti}\beta_{ktj}=\frac{1}{g_k(X_{\bar z})}
\sum_e(-1)^{\eta_e}f_{k,e}(X_{\bar z})X_{\bar z}^e.
$$
\: Note that the $d_{kij}$ in the above expression is coming from the representation in the basis $\bar e$ given above. The right hand side of the expression is coming from the fact that the sequence elements satisfy this instance of $\Phi$. But the linear independence of the vector sort elements forces the equality of the coefficients, even though it isn't a priori obvious that that the degree in the coefficient expressions of the left hand side have the correct shape (or e.g. low enough degree). After expanding both sides in the same ordered monomials with exponents zero or one, linear independence over the field of squares forces the collected coefficients on the left to equal those on the right. 

\: The point of central importance here is that the sign re-calculations which might occur with a sign change due to a reordering of the graph vertices in this expression above preserves these cancellations - substituting the corresponding squares preserves the identities between rational functions of the squares. The entries of $\bar z$ form exactly the support of this coefficient tuple: $S$ was its least support, and the allowed sign changes preserve it. On the other hand, the expression $\sum_t\alpha_{ti}\beta_{ktj}$ uses only generators indexed by $G_0\cup T_i\cup T_j$. Every entry of $\bar z$ therefore belongs to this union. 

\: Replace the entries of $\bar z$ in $T_i$ by their corresponding vertices in $T_j$, replace those in $T_j$ by their corresponding vertices in $T_i$, and leave those in $G_0$ fixed. Call the resulting tuple $\bar z'$. Its entries are still distinct.

\: For each $e$, there are three potential sign change sources to account for: the sign from
putting $X_{\bar z}^e$ in the order $G_0<T_j<T_i$, the sign from
the product calculation above, and the sign from putting the
substituted monomial in the order specified by $\bar z'$. Each depends on $e$, but does not depend on $k$. The identities between the rational coefficients remain true after substituting
the corresponding squares (since such sign changes do not affect rational expressions in the squares), so the product of these signs is
the only change. Choose $\eta'_e=\eta_e$ if this
product is positive, and $\eta'_e=1-\eta_e$ if it is negative (we want to get another formula which is a disjunct in $\Phi$).
Then
$$
d_{kji}
=\sum_{t=1}^r\alpha_{tj}\beta_{kti}
=\frac{1}{g_k(X_{\bar z'})}
\sum_e(-1)^{\eta'_e}f_{k,e}(X_{\bar z'})X_{\bar z'}^e.
$$
\: The polynomials $g_k,f_{k,e}$ are exactly those defining $\Phi$.
We have changed only the graph vertices substituted into them
along with the choice of signs which moves us to a different disjunct in $\Phi$. Substituting these coefficients in
$a_j=\sum_{k=1}^n d_{kji}b_{ki}$ and multiplying by the common
denominator gives the equation in the disjunct indexed by
$\eta'$. Thus $\bar z'$ witnesses $\Phi(a_j,\bar b_i)$,
and this contradicts our assumption since $j>i$.

\: This rules out the proposed indiscernible sequence in the case that the vector sort tuples are independent. By compactness there is a then finite conjunction $\rho(\bar y)$ of formulas in the complete type of $\bar w$ over $\emptyset$ which makes
$$
\rho(\bar y)\wedge\Phi(x,\bar y)
$$
stable. It isn't hard to see that this formula still forks over $B$. By our argument, we see: 

\begin{prop} 
The conclusion of the stable forking conjecture holds for $1$-types in $T^{\neg \mathrm{StFk}}$.
\end{prop}

\end{tcolorbox}

\bibliographystyle{plain}
\bibliography{refs}

\end{document}